\documentclass[11pt,reqno]{amsart}

\usepackage{geometry}
\usepackage[T1]{fontenc}
\usepackage[hidelinks]{hyperref} 

\usepackage{breakcites}

\usepackage{algorithm}
\usepackage{algpseudocode}
\usepackage{graphicx}
\usepackage{dsfont}
\usepackage{setspace}
\usepackage{enumitem}
\usepackage[dvipsnames]{xcolor}
\usepackage{graphicx}
\usepackage{epstopdf}
\usepackage{csquotes}
\usepackage{mathtools}
\usepackage{amssymb, amsmath,amsthm, amsfonts}
\usepackage{ifthen}

\usepackage{caption}

\usepackage{verbatim}
\usepackage{pifont}
\usepackage{bm}  
\usepackage{stackengine}
\usepackage{bbm}

\allowdisplaybreaks
\usepackage{accents}
\usepackage[capitalise]{cleveref}
\usepackage{xpatch}
\makeatletter
\xpatchcmd{\@thm}{\thm@headpunct{.}}{\thm@headpunct{}}{}{}

\crefname{equation}{Equation}{equations}
\Crefname{equation}{Equation}{Equations}

\theoremstyle{plain}
\newtheorem{theorem}{Theorem}[section]
\newtheorem{proposition}{Proposition}[section]
\newtheorem{lemma}{Lemma}[section]
\newtheorem{corollary}{Corollary}[section]
\newtheorem{assumption}{Assumption}[section]

\theoremstyle{definition}
\newtheorem{definition}{Definition}[section]
\newtheorem{example}{Example}[section]

\newtheoremstyle{italicremark}
  {}{}                 
  {\itshape}           
  {}                   
  {\bfseries}          
  {.}                  
  { }                  
  {}                   

\theoremstyle{italicremark}
\newtheorem{remark}{Remark}

\usepackage{graphicx}
\usepackage{subcaption}
\usepackage{enumitem}
\usepackage{url}
\usepackage{wrapfig,framed}
\usepackage{setspace}
\usepackage{graphicx}
\usepackage{epstopdf}
\usepackage{csquotes}
\usepackage{mathtools}
\usepackage{amssymb, amsmath,amsthm, amsfonts}
\usepackage{ifthen}
\usepackage{verbatim}
\usepackage{pifont}
\usepackage{bm}  
\usepackage{stackengine}
\usepackage{bbm}

\allowdisplaybreaks
\usepackage{accents}

\usepackage{xpatch}
\makeatletter
\xpatchcmd{\@thm}{\thm@headpunct{.}}{\thm@headpunct{}}{}{}
\makeatother

\usepackage{tikz}
\usetikzlibrary{positioning}

\DeclareMathOperator{\E}{\mathbb{E}}

\DeclareMathOperator{\id}{id}

\DeclareMathOperator{\argmin}{argmin}

\DeclareMathOperator{\supp}{spt}

\newcommand{\R}{\mathbb{R}}

\newcommand{\cB}{\mathcal{B}}

\newcommand{\cI}{\mathcal{I}}

\newcommand{\cL}{\mathcal{L}}

\newcommand{\cN}{\mathcal{N}}
\newcommand{\cO}{\mathcal{O}}
\newcommand{\cP}{\mathcal{P}}

\newcommand{\sD}{\mathsf{D}}

\newcommand{\sF}{\mathsf{F}}

\newcommand{\sH}{\mathsf{H}}

\newcommand{\sI}{\mathsf{I}}

\newcommand{\sR}{\mathsf{R}}

\newcommand{\sV}{\mathsf{V}}
\newcommand{\sW}{\mathsf{W}}

\newcommand{\NN}{\mathbb{N}}

\newcommand{\RR}{\mathbb{R}}

\newcommand{\op}{\mathrm{op}}
\newcommand{\F}{\mathrm{F}}

\newcommand{\llangle}{\left\langle}
\newcommand{\rrangle}{\right\rangle}

\newcommand{\tr}{\mathsf{tr}}

\newcommand{\dom}{\mathrm{Dom}}

\newcommand{\ac}{\mathrm{ac}}

\newcommand{\IGW}{\mathsf{IGW}}

\newcommand{\W}{\mathsf{W}}

\definecolor{darkblue}{rgb}{0.0,0.0,0.66}  
\definecolor{darkred}{rgb}{100,0.0,0.0} 
\newcommand{\KL}{\sH(\cdot\|\gamma)}

\newcommand{\zg}[1]{{\color{blue}[ZG: #1]}}

\newcommand{\vk}[1]{{\color{orange}[VK: #1]}}


\begin{document}

\thanks{
Z. Goldfeld is partially supported by NSF grants CCF-2308446 and CCF-2046018. K. Kato is partially supported by NSF grant DMS-2413405.}

\author[V. Karumanchi]{Venkatkrishna Karumanchi}
\address[V. Karumanchi]{
School of Operations Research and Information Engineering, Cornell University.
}
\email{vk383@cornell.edu}

\author[Z. Goldfeld]{Ziv Goldfeld}
\address[Z. Goldfeld]{
School of Electrical and Computer Engineering, Cornell University.
}
\email{goldfeld@cornell.edu}

\author[K. Kato]{Kengo Kato}
\address[K. Kato]{
Department of Statistics and Data Science, Cornell University.
}
\email{kk976@cornell.edu}

\author[Z. Zhang]{Zhengxin Zhang}
\address[Z. Zhang]{
Xantium Group.
}
\email{zz658@cornell.edu}

\title[Probabilistic Representation and Convergence of GW Gradient Flows]{Probabilistic Representation and Convergence of Gromov-Wasserstein Gradient Flows}
\begin{abstract}
Wasserstein gradient flows are intimately connected with evolution partial differential equations and diffusion processes. We take the first step in developing such connections for inner product Gromov--Wasserstein (IGW) gradient flows by studying the IGW gradient flow of the relative entropy $\KL$ with respect to the standard Gaussian measure $\gamma$. We first show that $\KL$ fails to be $\lambda$-convex along generalized or modified generalized IGW geodesics for any $\lambda \in \R,$ and therefore falls outside the scope of the existing IGW gradient flow theory from \cite{zhang2024gradient}. We bridge this gap by establishing a suitable \emph{local} convexity estimate that enables the construction of the gradient flow and its extension to the infinite time horizon. We then obtain increasingly explicit representations of the resulting dynamics. Starting from a partial integro-differential equation, we derive a nonlinear Fokker--Planck equation and show that its second-moment dynamics decouple from the law as they satisfy an autonomous matrix ODE. This reduces the IGW dynamics to a linear, time-inhomogeneous Fokker--Planck equation, yielding a probabilistic representation as the time-marginal flow of a linear stochastic differential equation resembling the Ornstein--Uhlenbeck process. Finally, we study its asymptotic behavior by establishing exponential convergence of the flow to $\gamma$ in relative entropy. 
\end{abstract}
\maketitle

\section{Introduction}

Since the seminal works of Jordan, Kinderlehrer, and Otto  \cite{jordan1998variational,otto2001geometry}, the interplay between gradient flows over the space of probability measures, evolution partial differential equations (PDEs), and associated stochastic differential equations (SDEs) has received considerable interest. The variational perspective provides a unified framework for the analysis of PDEs (see \cite{ambrosio2005gradient} and \cite{santambrogio2015optimal} for overviews) and has found numerous applications, such as numerical schemes for solving PDEs \cite{benamou2016augmented,peyre2015entropic}, training dynamics for neural networks \cite{mei2019mean,chizat2018global}, and sampling from probability distributions \cite{wibisono2018sampling,chewi26log}. The original construction in \cite{jordan1998variational}, now often referred to as the \textit{JKO scheme}, defines a sequence of probability measures by applying a proximal point method to a given energy functional with respect to (w.r.t.) the $2$-Wasserstein ($\sW_2$) metric, and then taking a suitable limit of the discrete sequence as the time step converges to zero.

Recent interest lies in extending the JKO scheme to discrepancy measures other than $\sW_2$, where the resulting gradient flows exhibit different dynamics that reflect the induced geometry \cite{duncan2023geometry,carrillo2026fisher,zhang2024gradient,rankin2026jko, hardion2025gradient}. Among these works, \cite{zhang2024gradient} provided the first construction of gradient flows in the inner product Gromov--Wasserstein (IGW) geometry, which compares distributions by minimizing distortions of pairwise inner products, rather than pointwise transport costs in a common ambient space, as in the Wasserstein distance. Formally, for $\mu,\nu\in\cP_2(\R^{d})$, the space of Borel probability measures on $\R^d$ with finite second moments, the IGW distance is defined as
\begin{equation}
\mathsf{IGW}(\mu,\nu)
\coloneqq
\left(
\inf_{\pi\in\Pi(\mu,\nu)}
\int_{(\R^d \times \R^d)^2}
\left| \langle x,x'\rangle - \langle y,y'\rangle
\right|^2 d\pi \otimes \pi(x,y,x',y')
\right)^{\frac12},
\label{eq:IGW_intro}
\end{equation}
which gives rise to a pseudometric on $\cP_2(\R^d)$ that is invariant under orthogonal transformations {\cite[Proposition 3.1]{zhang2024gradient}}. Here $\Pi(\mu,\nu)$ denotes the collection of couplings for $(\mu,\nu)$. IGW belongs to the broader family of Gromov--Wasserstein (GW) discrepancies, which are motivated by applications in fields such as computer graphics, linguistics, and biology where one may wish to identify distributions that agree up to transformations preserving their internal geometry \cite{solomon2016entropic,xu2019gromov,nitzan2019gene,alvarezmelis2018gromov}. The theory in \cite{zhang2024gradient} applies to rotationally invariant functionals satisfying a global $\lambda$-convexity condition along generalized or modified generalized IGW geodesics (see \cref{assumption:zz_main}~(iii) below) for some $\lambda \in \R$,  and characterized IGW gradient flows as distributional solutions to certain partial integro-differential equations (PIDEs). 

The framework of \cite{zhang2024gradient} leaves several fundamental questions open.
First, while the restriction to rotationally invariant functionals is natural from the point of view of adapting the JKO scheme to the IGW geometry (see \eqref{eq:minimizing_movement} for the definition of the IGW--JKO scheme), it remains unclear whether canonical functionals such as relative entropy satisfy $\lambda$-convexity along (modified) generalized IGW geodesics. This contrasts with the Wasserstein setting, where broad classes of functionals are known to satisfy $\lambda$-convexity along Wasserstein geodesics; see \cite[Chapter 9]{ambrosio2005gradient}. Second, the PIDE characterization in \cite{zhang2024gradient} is implicit, making the resulting dynamics difficult to interpret. In the Wasserstein setting, the gradient flow of relative entropy w.r.t. a log-concave target distribution is governed by a linear Fokker--Planck equation (FPE), which is the time-marginal evolution of the associated Langevin diffusion. Such PDE/SDE representations and their interplay are largely unexplored in the IGW case. Finally, the long-time behavior of IGW gradient flows, in particular, the convergence to equilibrium and quantitative convergence rates, was not addressed in \cite{zhang2024gradient}.

\subsection{Contributions}
In this work, we address these questions for the relative entropy functional, $\sH(\cdot\|\gamma)$, w.r.t. the standard Gaussian $\gamma=\cN(0,I)$, which is arguably one of the most fundamental rotationally invariant energy functionals (see \eqref{eq: rel entropy} below for the definition of the relative entropy). Our contributions are summarized as follows. 

First, we provide counterexamples showing that $\sH(\cdot \| \gamma)$ is \textit{not} $\lambda$-convex, for any $\lambda \in \R$, along either generalized or modified generalized IGW geodesics. As such, the general theory of \cite{zhang2024gradient} does not directly apply to our setting.
Nonetheless, we establish that the functional $\sH(\cdot \| \gamma)$ admits a suitable \emph{local} convexity estimate along modified generalized IGW geodesics, which allows us to adapt the proof of Theorem 4.1 in \cite{zhang2024gradient} to construct the IGW gradient flow for $\sH(\cdot \| \gamma)$. In addition, by establishing the uniform nondegeneracy of the second moment matrix flow, we extend the IGW gradient flow to the infinite time horizon, $(\rho_t)_{t \in [0,\infty)}$. 

Next, we provide analytic and probabilistic characterizations of the flow $(\rho_t)_{t \in [0,\infty)}$.  The limiting curve of probability measures constructed by the IGW--JKO scheme satisfies the continuity equation with a velocity field given through the action of the \textit{inverse} mobility operator  (see \cref{def:mob_op}) on the limiting subdifferential of $\KL$. Since the mobility operator is nonlocal, the resulting continuity equation for $(\rho_t)_{t\in[0,\infty)}$ initially takes the form of a PIDE. We evaluate the action of the inverse mobility operator in closed form, which gives rise to a \textit{nonlinear} FPE whose coefficients depend on the second moment matrix $\Sigma_{\rho_t}$ of $\rho_t$. The derivation requires some care, as little a priori regularity is available for $\rho_t$. The key structural observation is that the second moment matrix flow, $\Sigma_{\rho_t}$, satisfies an autonomous ordinary differential equation (ODE), which allows us to recast the nonlinear FPE as a \emph{linear, time-inhomogeneous} FPE. The latter linear FPE is realized as the time-marginal flow of a certain linear SDE that resembles the Ornstein--Uhlenbeck (OU) process, thereby providing a simple probabilistic representation of the IGW gradient flow for $\sH(\cdot \| \gamma)$. In comparison, the classical OU process generates the Wasserstein gradient flow of $\sH(\cdot \| \gamma)$.

Finally, we study the long-time behavior of the flow. Using the Gaussian log-Sobolev inequality together with uniform spectral bounds on $\Sigma_{\rho_t}$ along the flow, combined with techniques from \cite{ambrosio2007gradient}, 
we establish exponential convergence of the flow to $\gamma$ in relative entropy, which is reminiscent of the classical OU semigroup (cf. \cite{bakry2014analysis}) despite the difference in the underlying geometries.

\subsection{Related literature}

There is a vast literature on the JKO scheme, Wasserstein gradient flows, and their applications. We refer the reader to the textbooks and surveys \cite{ambrosio2005gradient,santambrogio2015optimal,santambrogio2017euclidean,peyre2019computational,chewi2025statistical,chewi26log}, and the references therein, for broader background and exposition. There is a recent line of literature exploring gradient flows over the space of probability measures endowed with discrepancy measures other than Wasserstein metrics. Apart from the IGW distance, they include the Stein distance \cite{duncan2023geometry}, the Fisher-Rao distance \cite{carrillo2026fisher}, the general transport cost \cite{rankin2026jko}, and the Sinkhorn divergence \cite{hardion2025gradient}.

The GW distance, first proposed by \cite{memoli2011gromov}, defines a metric on the space of all Polish metric measure spaces modulo measure-preserving isometries (see also \cite{sturm2012space}) and enables comparing and aligning heterogeneous and structured data sets. The specific variant of the IGW distance was considered by \cite{vayer2020contribution}. Subsequent work has studied GW and IGW from computational, statistical, and application-oriented perspectives. Along with the previously mentioned applications, GW-based methods have been applied to computational biology, including single-cell multi-omics integration \cite{demetci2022scot,cao2022manifold}, spatial reconstruction from single-cell transcriptomic data \cite{nitzan2019gene}, and the alignment and integration of spatial transcriptomics data \cite{zeira2022alignment,liu2023partial}. Recently, the IGW distance has also been leveraged for machine learning tasks such as machine translation \cite{alvarezmelis2018gromov}, latent correspondence learning \cite{le2022entropic}, LLM representation comparison and heterogeneous clustering of text data \cite{dandapanthula2025optimal}, among others. Motivated by these applications, several works have investigated the computational and statistical properties of GW and IGW; see, e.g., \cite{scetbon2022linear, zhang2024gromov,rioux2024entropic,kato2025convergence, karumanchi2025approximation,gong2026sliced}.

As previously mentioned, the study of GW gradient flows was initiated by \cite{zhang2024gradient}, on which the present paper builds. The relative entropy functional considered here, however, falls outside the scope of their general theory, necessitating a separate construction of the corresponding IGW gradient flow, which in turn provides the foundation for the subsequent PDE and SDE analysis.

\subsection{Organization}
The rest of the paper is organized as follows. \cref{sec: background} presents a brief overview of the necessary tools from variational analysis, the IGW distance, and its gradient flow framework. The section concludes with counterexamples showing that the $\sH(\cdot \| \gamma)$ functional fails to satisfy $\lambda$-convexity, for any $\lambda \in \R$, along (modified) generalized IGW geodesics. In \cref{sec: construction}, we establish the existence of the IGW gradient flow for $\KL$ and extend the construction to the infinite time horizon. \cref{sec: main} presents the main results, where we first derive the nonlinear and linear FPEs for the IGW gradient flow, and then establish exponential convergence in relative entropy.
Sections \ref{sec: proof construction} and \ref{sec: proof main} contain the proofs of the results in Sections \ref{sec: construction} and \ref{sec: main}, respectively. \cref{sec: discussion} provides concluding remarks and discusses future research directions. Finally, the appendices contain additional auxiliary results and their proofs.

\section{Background and Preliminaries}\label{sec: background}
\subsection{Notation} 
We use $\| \cdot \|$ and $\langle \cdot, \cdot \rangle$ to denote the standard Euclidean norm and inner product for vectors, respectively. For any symmetric matrix $A$, let $\lambda_{\min}(A)$ and $\lambda_{\max}(A)$ denote its minimum and maximum eigenvalues, respectively. For any positive semidefinite (PSD) matrix $\Sigma \in \R^{d \times d}$, we use $\gamma_{\Sigma}$ to denote the Gaussian distribution with mean zero and covariance matrix $\Sigma$. When $\Sigma=I$, we write $\gamma = \gamma_{I}$.

We denote by $\cP(\R^d)$ the space of Borel probability measures on
$\R^d$, and by
$
\cP_2(\R^d)
\coloneqq
\left\{
    \mu\in\cP(\R^d)
    :
    \int_{\R^d}\lVert x\rVert^2\,d\mu(x)<\infty
\right\}
$
the space of probability measures with finite second moments. We set
$
\cP_2^{\mathrm{ac}}(\R^d)
\coloneqq
\{ \mu \in \cP_2(\R^d) : \mu \ll \cL^d \},
$
where $\cL^d$ is the Lebesgue measure on $\R^d$. 
Whenever $\mu \ll \cL^d$, we identify $\mu$ with its Lebesgue density $\frac{d\mu}{d\cL^d}$. 
For $\mu,\nu\in\cP(\R^d)$,
we use $\Pi(\mu,\nu)$ to denote the set of
couplings for $\mu$ and $\nu$, namely, the probability measures on
$\R^d\times\R^d$ whose first and second marginals are $\mu$ and $\nu$,
respectively. More generally, for $\mu_1,\dots,\mu_k \in \cP(\R^d)$, $\Pi(\mu_1,\ldots,\mu_k)$ denotes the set of
probability measures on $(\R^d)^{k}$ whose $i$th marginal  agrees with $\mu_i$. For $\mu,\nu\in\cP_2(\R^d)$, their $2$-Wasserstein
distance is defined by
\[
\W_2(\mu,\nu)
\coloneqq
\left(
    \inf_{\pi\in\Pi(\mu,\nu)}
    \int_{\R^d\times\R^d}
    \lVert x-y\rVert^2\,d\pi(x,y)
\right)^{1/2}.
\]
For $\mu\in\cP_2^{\mathrm{ac}}(\R^d)$ and
$\nu\in\cP_2(\R^d)$, we denote by $T^{\mu\to\nu}\colon\R^d\to\R^d$ the Brenier map transporting $\mu$ to $\nu$ \cite{brenier1991polar}.  For $\mu,\nu \in \cP(\R^d)$, the relative entropy $\sH(\mu\|\nu)$ is defined by
\begin{equation}
\sH(\mu\|\nu)
\coloneqq
\begin{cases}
\int_{\R^d}
\log\left(\frac{d\mu}{d\nu}\right)
\,d\mu,
& \text{if }\mu\ll\nu,\\[1em]
+\infty,
& \text{otherwise}.
\end{cases}
\label{eq: rel entropy}
\end{equation}
For $\mu \in \cP(\R^d)$, let $L^2(\mu;\R^d)$ denote the (real) Hilbert space of Borel vector fields $v: \R^d \to \R^d$ such that $\| v \|_{L^2(\mu; 
R^d)}^2 = \int \|v\|^2 \, d\mu<\infty$, endowed with the inner product $\langle v,w \rangle_{L^2(\mu;\R^d)} = \int_{\R^d} \langle v,w \rangle \, d\mu$.
We use $\Sigma_\mu$ to denote the second moment matrix of $\mu$, i.e., $\Sigma_\mu = \int xx^\intercal \, d\mu(x)$.
In addition, we denote $M_2(\mu) = \int \|x\|^2 \, d\mu(x) = \tr (\Sigma_\mu)$. 

We use standard function space notations ($C_c^\infty, W^{1,1}, W^{1,1}_{\mathrm{loc}},$ etc.); see \cite{ambrosio2005gradient}. For an absolutely continuous map $t \in [0,T] \mapsto A_t \in \R^{d \times d}$, we write $\dot{A}_t$ for its derivative with respect to $t$, whenever it exists.
Finally, we use $\lesssim_\xi$ to denote inequalities up to constants that only depend on $\xi$; the subscript is dropped when the constant is universal.

\subsection{Subdifferential calculus}
We recall some basic definitions from subdifferential calculus in the space of probability measures.

\begin{definition}[Fr\'echet subdifferential]\label{definition:frechet_subdifferential}
    Given a proper and lower semicontinuous functional $\sF:\cP_2(\RR^d)\to\RR\cup\{+\infty\}$ with $\dom(\sF) \coloneqq\{\mu\in\cP_2(\RR^d):\sF(\mu)<+\infty\} \subset\cP_2^{\ac}(\RR^d)$, its \textit{Fr\'echet subdifferential} $\partial\sF(\mu)$ is the collection of all $\xi\in L^2(\mu;\RR^d)$ such that
    \begin{align*}
        \sF(T_{\#}\mu)-\sF(\mu)\geq\int_{\R^d} \llangle\xi(x),T(x)-x\rrangle d\mu(x) +o(\|T-\id\|_{L^2(\mu;\RR^d)})
    \end{align*}
    for any map $T \in L^2(\mu;\R^d)$. Any element $\xi\in\partial\sF(\mu)$ is called a \textit{strong subdifferential}.
\end{definition}

We shall make heavy use of the \textit{limiting subdifferential}, which we define next.  Denote $\dom(\partial \sF) = \{\mu \in \cP_2(\R^d): \partial \sF(\mu) \neq \emptyset\}$ for the domain of the Fr\'echet subdifferential.

\begin{definition}[Limiting subdifferential]
For $\mu \in \dom(\sF)$, we say that a vector field
$\xi \in L^2(\mu;\R^d)$ belongs to the \textit{limiting subdifferential}
$\partial_{\ell}\sF(\mu)$ of $\sF$ at $\mu$ if there exist two
sequences
$\mu_k \in \dom(\partial\sF)$ and 
$\xi_k \in \partial\sF(\mu_k)$ such that $\mu_k \to \mu$ weakly in $\cP(\R^d)$, $\xi_k \to \xi$ weakly, in the sense that $\lim_{k} \int_{\R^d}\langle g,\xi_k \rangle \, d\mu_k = \int_{\R^d} \langle g,\xi \rangle \,d\mu$ for all $g \in C_c^\infty(\R^d;\R^d)$, and 
\[
\sup_k\left(\sF(\mu_k)+\int_{\RR^d} \left(\|x\|^2 + \|\xi_k(x)\|^2 \right)\,d\mu_k(x)\right) <+\infty.
\]
\end{definition}

\subsection{Inner product Gromov--Wasserstein distance}\label{subsec:IGW}

We consider the GW distance with inner product cost (abbreviated IGW) between  two Euclidean metric measure spaces $(\RR^{d},\|\cdot\|,\mu)$ and $(\RR^{d},\|\cdot\|,\nu)$ for $\mu,\nu \in \cP_2(\R^d)$, defined by \eqref{eq:IGW_intro}. 
The infimum of the right-hand side of \eqref{eq:IGW_intro} is achieved for some $\pi \in \Pi(\mu,\nu)$ (see Lemma 1.2 in \cite{sturm2012space}), which we call an optimal IGW coupling. 
By definition, the IGW distance is rotationally invariant, in the sense that, for any orthogonal matrix $O$, $\IGW(O_{\#}\mu,\nu)=\IGW(\mu,O_{\#}\nu) = \IGW(\mu,\nu)$. As such, it is only a pseudometric on $\cP_2(\R^d)$, and $\mathsf{IGW}(\mu,\nu) = 0$ if and only if there exists a mapping $T: \supp (\mu) \to \supp(\nu)$ that preserves the inner product $\mu \otimes \mu$-a.s. such that $\nu = T_{\#}\mu$, where $\supp(\mu)$ denotes the support of $\mu$; see Proposition 3.1 in~\cite{zhang2024gradient}.

\subsection{IGW gradient flows}\label{subsec:IGW flows}

We review the main result from \cite{zhang2024gradient} on the construction of IGW gradient flows, which builds on adapting the JKO scheme to the IGW distance. 
The IGW--JKO schemes considered in \cite{zhang2024gradient} require a certain subset of the orthogonal group in $\RR^d$, using which cross-covariance matrices induced by optimal IGW couplings can be symmetrized. 

\begin{definition}[Cross-covariance PSD transform]\label{def:PSD_transform}
    Fix $\mu,\nu\in\cP_2(\RR^d)$ and let $\pi^\star\in\Pi(\mu,\nu)$ be an optimal IGW coupling. Consider the singular value decomposition of the cross-covariance matrix $\int xy^\intercal d\pi^\star(x,y) = P\Lambda Q^\intercal$, where $P, Q$ are orthogonal matrices and  $\Lambda$ is a diagonal matrix, and let $O=PQ^\intercal$. Define $\cO_{\mu,\nu}$ to be the collection of such orthogonal matrices $O$, from any optimal IGW coupling. 
\end{definition}

Next, we introduce the IGW--JKO scheme. Consider a functional $\sF$ that is rotationally invariant, weakly lower semicontinuous, and regular in the sense of Definition 10.1.4 in \cite{ambrosio2005gradient}, with $\dom(\sF) \subset \cP_2^{\ac}(\R^d)$. For given $\delta > 0,\; n\in\NN$, and $\rho_0 \in \cP_2(\R^d)$ with $\sF(\rho_0) < \infty$ and nonsingular $\Sigma_{\rho_0}$, the IGW--JKO scheme with step size $\tau = \frac{\delta}{n}$ generates a sequence of probability measures $\{\rho_i^\tau\}_{i=0}^n\subset\cP_2(\RR^d)$ with $\rho_0^\tau=\rho_0$ defined via 
\begin{equation}
\begin{split}
    &\tilde\rho_{i+1}^\tau \in \argmin_{\rho\in\cP_2(\RR^d)} \sF(\rho) + \frac{1}{2\tau}\IGW^2(\rho,\rho_i^\tau),\\
    &\rho_{i+1}^\tau = O_{\#}\tilde{\rho}_{i+1}^\tau, \ \ O\in\cO_{\rho_i^\tau,\tilde{\rho}_{i+1}^\tau}.
\end{split}
\quad i =0,\dots,n-1,
\label{eq:minimizing_movement}
\end{equation}
This gives rise to the piecewise constant curve of probability measures $(\bar{\rho}^n_t)_{t \in [0,\delta]}$,
\[
\bar{\rho}^n_t=\rho_i^\tau, \quad t\in ((i-1)\tau,i\tau], \quad i=1,\ldots,n
\]
with $\bar \rho^n_0 = \rho_0.$ The rotation step in (\ref{eq:minimizing_movement}) does not feature in the JKO scheme for the Wasserstein distance. 
Formally, the rotation step aligns $\rho_{i+1}^\tau$ with $\rho_{i}^\tau$ in $\W_2$, resulting in a sequence $\{\rho_i^\tau\}_{i=0}^n$ that complies with the natural Wasserstein gradient flow structure associated with the continuity equation. More precisely, the rotation step is essential to first ensure that the discrete-time velocity field is induced by a Gromov--Monge map (see \cref{lem:gromov_monge} below) and to then identify the limiting velocity field.

The original analysis in \cite{zhang2024gradient} focused on functionals that are $\lambda$-convex either along \textit{generalized IGW geodesics} \cite[Definition 3.2]{zhang2024gradient}, defined analogously to generalized geodesics in $\W_2$ \cite[Definition 4.3]{santambrogio2017euclidean}, or along \textit{modified generalized IGW geodesics}. For our application to $\sH(\cdot \| \gamma)$, we will work with modified generalized geodesics, whose definition we recall next, starting from the following lemma.  

\begin{lemma}[Gromov--Monge map {\cite[Lemma 2.2]{zhang2024gradient}}]\label{lem:gromov_monge}
    For $\mu\in\cP_2^{\mathrm{ac}}(\RR^d)$ and $\nu\in\cP_2(\RR^d)$, let $\pi^\star\in\Pi(\mu,\nu)$ be an optimal IGW coupling such that $A^\star=\frac{1}{2}\int xy^\intercal d\pi^\star(x,y)$ is nonsingular. Then $T^\star\coloneqq (8A^\star)^{-1}T^{\mu\to (8A^\star)_{\#}\nu}$ is a Gromov--Monge map for $\IGW(\mu,\nu)$, i.e., we have $\pi^\star = (\id,T^\star)_{\#}\mu$ and $A^\star = \frac{1}{2}\int x\, T^\star(x)^\intercal d\mu(x)$.
\end{lemma}
The definition of modified generalized IGW geodesics is given below. 
\begin{definition}[Modified generalized IGW geodesics  \cite{zhang2024gradient}]\label{def:MGG}
    For  $\mu_0,\mu_1,\mu_2\in\cP^{\ac}_2(\RR^d)$, let $\pi_{01}$ and $\pi_{02}$ be optimal IGW couplings for $(\mu_0,\mu_1)$ and $(\mu_0,\mu_2)$, respectively. Suppose that $\mu_1, \mu_2$ are already rotated w.r.t. $\mu_0$ in the sense of \cref{def:PSD_transform}, so that both optimal couplings have symmetric positive definite (uncentered) cross-covariance matrices $A_1 \coloneqq \frac 12 \int xy^\intercal d\pi_{01}(x,y)$ and $A_2 \coloneqq \frac 12 \int xz^\intercal d\pi_{02}(x,z)$. Consider
    \begin{equation}
    \label{eq: map}
    T_t \coloneqq \Bigl((1-t)(8A_1)^{-1}+t(8A_2)^{-1}\Bigr)
    \Bigl((1-t)\nabla \phi_1+t\nabla \phi_2\Bigr),
    \end{equation}
where $\phi_1$ and $\phi_2$ are convex functions such that $T_i\coloneqq (8A_i)^{-1}\nabla\phi_i$ is a Gromov--Monge map from $\mu_0$ to $\mu_i$, for each $i=1,2$ (cf. \cref{lem:gromov_monge}). 
The \textit{modified generalized IGW geodesic} between $\mu_1$ and $\mu_2$ w.r.t. $\mu_0$ is defined by $\nu_t = (T_t)_{\#} \mu_0$ for $t\in[0,1]$.
\end{definition}

Describing the limiting dynamics arising from the IGW--JKO scheme requires the following \textit{mobility operator} from \cite{zhang2024gradient}.

\begin{definition}[Mobility Operator \cite{zhang2024gradient}]\label{def:mob_op}
For any $\mu\in\cP_2(\RR^d)$, the \textit{mobility operator} $\cL_{\mu}:L^2(\mu;\RR^d)\to L^2(\mu;\RR^d)$ is defined by
\begin{equation}
    \cL_{\mu}[v](x) \coloneqq 2\Sigma_{\mu}  v(x) + 2 \left ( \int_{\RR^d} y  v(y)^\intercal \, d\mu(y)\right) x , \ v \in L^2(\mu;\R^d). 
    \label{eq:operator}
\end{equation}
\end{definition}

We present some properties of the mobility operator that will be used in the sequel. 
\begin{proposition}[Properties of $\cL_{\mu}$]\label{prop:fredholm}
The following hold.
\begin{enumerate}
    \item[(i)] The operator $\cL_{\mu}$ is linear, bounded, and self-adjoint with invariant subspace 
\[
    \cI_{\mu}
    \coloneqq
    \left\{
    v \in L^2(\mu;\RR^d)
    :
    \int_{\R^d} x v(x)^\intercal \,d\mu(x)
    \text{ is symmetric}
    \right\},
    \]
    which is closed in $L^2(\mu; \R^d)$. 
    If $\Sigma_\mu$ is nonsingular, then the kernel of $\cL_{\mu}$ is 
      \[ 
    \mathrm{ker}(\cL_\mu) = \left\{
    w \in L^2(\mu;\RR^d): w(x) = Bx \text{ for $\mu$-a.e. } x, \text{ with } B \in \RR^{d \times d}
    \text{ skew-symmetric}
    \right\}.
    \] 
    \item[(ii)] If $\Sigma_\mu$ is nonsingular, then the restricted operator $\cL_{\mu}|_{\cI_{\mu}}$ is a bijection onto $\cI_{\mu}$ with bounded inverse
\begin{equation}
(\cL_{\mu}|_{\cI_{\mu}})^{-1}[w](x) =\frac{1}{2} \Sigma_{\mu}^{-1} w(x)-\Sigma_\mu^{-1} B_{w} x, \quad w \in \cI_\mu, 
\label{eq: inverse}
\end{equation}
where $B_w \in \R^{d \times d}$ is the unique (and necessarily symmetric) solution to the Sylvester equation $\Sigma_\mu B + B\Sigma_\mu = \frac{1}{2}\int_{\R^d} yw(y)^\intercal \, d\mu$.  
\end{enumerate}
\end{proposition}

These properties are essentially proved in Section 10.2 of \cite{zhang2024gradient}. For the sake of completeness, we provide a short proof of the above proposition in \cref{sec: app A}.

Given these preparations, we recall the assumptions required for the main result (Theorem 4.1) in \cite{zhang2024gradient}.
\begin{assumption}\label{assumption:zz_main}
The initial distribution $\mu_0\in\dom(\sF)\subset\cP_2(\RR^d)$ has a nonsingular second moment matrix $\Sigma_{\mu_0}$ and the objective functional $\sF:\cP_2(\RR^d)\to\RR\cup\{+\infty\}$~is:
\begin{enumerate}
    \item[(i)] rotationally invariant, lower bounded, and weakly lower semicontinuous;
    \item[(ii)] regular in the sense of \cite[Definition 10.1.4]{ambrosio2005gradient} and has $\dom(\sF)\subset\cP_2^{\mathrm{ac}}(\RR^d)$.
    \item[(iii)] $\lambda$-convex along any generalized (or modified generalized) IGW geodesics with some $\lambda\in\RR$, i.e., for $\mu_0,\mu_1,\mu_2 \in \dom(\sF)$ with the glued joint distribution $\pi\in\Pi(\mu_0,\mu_1,\mu_2)$ and a generalized (or modified generalized) geodesic $\nu_t$ from $\mu_1$ to $\mu_2$  w.r.t. $\mu_0$, we~have
    \begin{equation}
    \begin{split}
    &\sF(\nu_t)\leq (1-t)\sF(\mu_1)+t\sF(\mu_2) \\
    &\quad -\lambda t(1-t) \int_{(\R^d \times \R^d)^2} \left|\llangle y,y'\rrangle-\llangle z,z'\rrangle\right|^2d\pi_{1,2}\otimes\pi_{1,2}(y,z,y',z'), \quad \forall t\in[0,1],
    \end{split}
    \label{eq:lambda_convexity}
    \end{equation}
    where $\pi_{1,2}$ denotes the $(y,z)$-marginal of $\pi$, i.e., $\pi_{1,2}  = \big ((x,y,z) \mapsto (y,z)\big)_{\#} \pi$.
\end{enumerate}
\end{assumption}

 The following is a slight variant of Theorem 4.1 in \cite{zhang2024gradient}.
\begin{theorem}[IGW gradient flow \cite{zhang2024gradient}]\label{thm:zz_main1}
Let $\rho_0\in\cP_2(\RR^d)$ and let
$\sF:\cP_2(\RR^d)\to(-\infty,+\infty]$ satisfy
Assumption \ref{assumption:zz_main} (i) and (ii). Set $\delta_{\rho_0} \coloneqq
    \frac{\left(1-1/\sqrt{2}\right)\lambda_{\min}(\Sigma_{\rho_0})}
    {2^{5/4}(\sF(\rho_0)-\sF_\star)},$ with $\sF_\star = \inf \sF$.
For any $\delta\in(0,\delta_{\rho_0})$, the following hold:
\begin{enumerate}
\item[(i)]
The piecewise-constant
IGW--JKO interpolant
$\bar\rho^n:[0,\delta]\to\cP_2(\RR^d)$, with time step
$\tau=\delta/n$, converges pointwise weakly to a $\W_2$-continuous curve
$\rho:[0,\delta]\to\dom(\sF)$, and the limiting curve satisfies the continuity equation
\begin{equation}\label{eq:zz_PIDE}
    \partial_t\rho_t+\nabla\cdot(\rho_t v_t)=0,
\end{equation}
in the distributional sense on $\RR^d\times(0,\delta),$ i.e.
\[  \int_0^\delta\int_{\RR^d} \partial_t g(t,x) \, d \rho_t(x)dt=-\int_0^\delta\int_{\RR^d} \llangle\nabla_x g(t,x), v_t(x)\rrangle d  \rho_t(x)dt,  \ \  \forall g\in C_c^\infty((0,\delta)\times\RR^d),
\]
for some velocity field $v_t \in L^2(\rho_t;\R^d).$
\item[(ii)] If, in addition, Assumption \ref{assumption:zz_main}~(iii) holds, the aforementioned convergence can be lifted to
uniform $\W_2$-convergence along a subsequence, and the corresponding velocity field $v_t$ satisfies
\begin{equation}
v_t\in-(\cL_{\rho_t}|_{\cI_{\rho_t}})^{-1}[\partial_{\ell}\sF(\rho_t)] \quad \text{$\rho_t$-a.s. for a.e. $t \in (0,\delta)$}.
\label{eq: differential inclusion}
\end{equation}
\end{enumerate}
\end{theorem}

\begin{remark}[Comparison with {\cite[Theorem~4.1]{zhang2024gradient}}]
\cref{thm:zz_main1} slightly differs from Theorem 4.1 in \cite{zhang2024gradient} but follows directly from its proof.
\cref{thm:zz_main1}~(i) constructs,
on the finite interval $[0,\delta]$, a limiting curve for the IGW--JKO scheme. The corresponding velocity field is obtained as a suitable limit of the discrete velocity fields, which satisfy the first-order optimality conditions associated with the discrete sequence $\{\rho_i^\tau\}_{i=0}^n$. This is the content of Propositions 4.1--4.4 from \cite{zhang2024gradient}. \cref{thm:zz_main1}~(ii) derives the gradient flow structure for the limiting curve $\rho$ by establishing that the velocity field $v_t$ constructed in (i) lies in the limiting subdifferential after the action of the mobility operator. In \cite{zhang2024gradient}, the extra $\lambda$-convexity assumption (\cref{assumption:zz_main}~(iii)) plays a crucial role in passing the optimality conditions for the discrete velocity fields to the continuous-time limit, which entails strengthening pointwise weak convergence to uniform $\sW_2$-convergence. These arguments are presented in Propositions 4.5--4.7 from \cite{zhang2024gradient}.

We note that Theorem 4.1~(2) in \cite{zhang2024gradient} states  $v_t \in -(\mathcal{L}_{\rho_t}|_{\mathcal I_{\rho_t}})^{-1}
[\partial \mathsf F(\rho_t)]$. The proof in their Proposition 4.7 invokes Theorem 11.1.6 in \cite{ambrosio2005gradient}, which, however, 
appears to yield only $v_t \in -(\cL_{\rho_t}|_{\cI_{\rho_t}})^{-1} [\partial_\ell \sF(\rho_t)]$.

Finally, Theorem 4.1~(2) in \cite{zhang2024gradient} focuses on the case where the (limiting) subdifferential is a singleton given by the gradient of the first variation of $\sF$, which turns the continuity equation \eqref{eq:zz_PIDE} into a PIDE. Abusing terminology, we shall call the continuity equation \eqref{eq:zz_PIDE} coupled with the differential inclusion \eqref{eq: differential inclusion} a PIDE. 
\end{remark}

\cref{assumption:zz_main}~(i) and (ii) hold for the relative entropy functional $\sH(\cdot \| \gamma)$, so \cref{thm:zz_main1}~(i) applies to our case. However, as shown by the counterexample below, \cref{assumption:zz_main}~(iii) fails to hold for $\sH(\cdot \| \gamma)$. This necessitates a separate argument to establish the conclusion of \cref{thm:zz_main1}~(ii) for $\sH(\cdot \| \gamma)$, which occupies the bulk of the proof of \cref{thm:main} below. The following example presents a case where $\sH(\cdot \| \gamma)$ fails to be $\lambda$-convex along modified generalized IGW geodesics for any $\lambda \in \R$. A counterexample for the generalized IGW geodesic case is lengthy and deferred to \cref{sec: counterexample}.

\begin{example}[Counterexample for modified generalized geodesics]
Let $\mu_0 =  \mu_1 = \cN(0,1),\; \mu_2 = \cN(0,r^2)$, and $X \sim \mu_0.$ One may check that $(X,X, rX)$ is a coupling for $(\mu_0,\mu_1,\mu_2)$ such that $(X,X)$ and $(X,rX)$ give optimal IGW couplings for $(\mu_0,\mu_1)$ and $(\mu_0,\mu_2)$, respectively. Then, $A_1 = \frac12$ and $A_2 = \frac{r}{2},$ which are strictly positive. Defining $\varphi_1(x) = 2x^2$ and $\varphi_2(x) = 2r^2 x^2,$ we have 
$T_1 = \frac{1}{8A_1}\nabla \varphi_1$ and $T_2 = \frac{1}{8A_2}\nabla \varphi_2.$
We may now consider the modified generalized geodesic $\nu_t = (T_t)_{\#}\mu_0$ 
for $t \in [0,1]$ with \eqref{eq: map}. A direct calculation shows that $\nu_{1/2} \sim \cN(0,q^2(r)),$ where $q(r) = \frac{(1 + r)(1 + r^2)}{4r}$.

Consider any $L > 0.$ Since $\int|xx' - r^2xx'|^2 \, d\mu_0(x)d\mu_0(x') = (1 - r^2)^2,$ to contradict $\lambda$-convexity for $\lambda = - L,$
we require that
\begin{equation}\label{eq:want}
    \sH(\nu_{1/2}\|\gamma) > \frac12\sH(\mu_{1}\|\gamma) + \frac12\sH(\mu_{2}\|\gamma) + \frac{L}{4}(1 - r^2)^2 = \frac12\sH(\mu_{2}\|\gamma) + \frac{L}{4}(1 - r^2)^2.
\end{equation}
Recall that $\sH(\cN(0,c^2)\|\gamma) = \frac12 (c^2 - 2\log c - 1).$ So \eqref{eq:want} is equivalent to
\begin{align*}
    q^2(r) - 2 \log q(r) - 1 > \frac{1}{2}\left (r^2 - 2 \log r - 1 + \frac{L}{2}(1 - r^2)^2 \right).
\end{align*}
Since $q^2(r) \sim \frac{1}{16r^2},$ and all other terms grow at most logarithmically to $\infty$ as $r\to0,$ the above inequality holds true for any $L > 0$ for small enough $r.$ Thus, by choosing $r$ appropriately, we have shown that $\KL$ is not $\lambda$-convex for any $\lambda \in \R.$
\end{example}

\section{IGW Gradient Flow for Relative Entropy}\label{sec: construction}

In this section, we formally construct IGW gradient flows for the relative entropy functional. As discussed in the previous section, the results in \cite{zhang2024gradient} do not fully cover $\sH(\cdot \| \gamma)$.
Therefore, our first goal is to establish the gradient flow structure from \cref{thm:zz_main1} for our functional and extend the flow to $[0,\infty)$. The latter extension is needed to consider the long-time behavior of the flow. We summarize the result as follows, with all the proofs for this section deferred to \cref{sec: proof construction}.

\begin{theorem}[IGW gradient flow for relative entropy]\label{thm:main}
Let $\rho_0\in\cP_2(\RR^d)$ have a nonsingular second moment matrix and  satisfy $\sH(\rho_0\|\gamma) < \infty$. Then:
\begin{enumerate}
\item[(i)] The conclusions of Theorem \ref{thm:zz_main1}~(i) and (ii) hold with $\sF = \sH(\cdot \| \gamma)$.
\item [(ii)] One can extend the curve $\rho$ from $[0,\delta]$ to $[0,\infty)$, so that it satisfies
\begin{equation}
\label{eq:PIDE}
\begin{cases}
&\partial_t\rho_t+\nabla\cdot(\rho_t v_t)=0 \quad \text{in the distributional sense on $\R^d \times (0,\infty)$}, \\
&v_t\in-(\cL_{\rho_t}|_{\cI_{\rho_t}})^{-1}[\partial_{\ell}\sH(\rho_t\|\gamma)] \quad \text{$\rho_t$-a.s. for a.e.
$t\in(0,\infty)$}. 
\end{cases}
\end{equation}
\end{enumerate}
\end{theorem}

Extending the conclusion of \cref{thm:zz_main1}~(ii) to our functional $\sH(\cdot \| \gamma),$ which fails to meet \cref{assumption:zz_main}~(iii), requires strengthening pointwise weak convergence of the IGW--JKO interpolant $\bar \rho^n$ to uniform $\sW_2$-convergence along a subsequence. This, in turn, requires a separate argument to obtain uniform Cauchy estimates in $\IGW$ for $\bar \rho^n$; see the discussion at the start of \cref{sec:construction_proof}. Our key observation is that, while $\sH(\cdot \| \gamma)$ fails to be $\lambda$-convex along (modified) generalized IGW geodesics, it nevertheless admits a `local' convexity estimate as stated next.

\begin{lemma}[Local convexity of $\KL$]\label{lem:local_convexity}
In the setting of \cref{def:MGG}, we have
\begin{multline*}
\sH(\nu_t \| \gamma)
\le (1-t)\sH(\mu_1 \| \gamma)+t\sH(\mu_2 \| \gamma) \\
+ t\,\frac{8\sqrt{2}}{c_{1,\mu_0}}\,\beta_{\pi}\,
     \int_{(\R^d \times \R^d)^2} \bigl|\langle y,y'\rangle-\langle z,z'\rangle\bigr|^2\, d\pi_{1,2}\otimes\pi_{1,2}(y,z,y',z')  \\
  + t\,\frac{12\sqrt{2}}{c_{1,\mu_0}}\,\beta_{\pi}\, \IGW^2(\mu_1,\mu_0)
  + O(t^2),
\end{multline*}
where $c_{1,\mu_0} \coloneqq \lambda_{\min}(\Sigma_{\mu_0}) > 0,$  $c_{2,\pi} > 0$ is chosen such that
$1/c_{2,\pi} \leq \lambda_{\min}(A_i) \leq \lambda_{\max}(A_i) \leq c_{2,\pi},$ for $i = 1,2$, and
\[
\beta_{\pi} = \frac{c_{2,\pi}}{2}
M_2^{1/2}(\mu_0)M_2^{1/2}(\mu_1) + \frac{c_{2,\pi}^4}{4} M_2(\mu_0)M_2(\mu_1).
\]
\end{lemma}

The proof of \cref{thm:main}~(ii) uses a concatenation argument to repeatedly initiate another gradient flow at the endpoint of the previous run. This argument requires establishing uniform nondegeneracy of the second moment matrices $\Sigma_{\rho_t}$ along the flow, which we account for using the following lemma.

\begin{lemma}[Spectral control from relative entropy]
\label{lem:spectral-control}
If $\sH(\rho\|\gamma)\le C<\infty$, then
\[
   \exp\!\left(-2C-d-(d-1)\log(4C+2d)\right) \leq \lambda_{\min}(\Sigma_\rho)
   \leq \lambda_{\max}(\Sigma_\rho) \leq 4C+2d .
\]
\end{lemma}

The uniform nondegeneracy of $\Sigma_{\rho_t}$ along the IGW gradient flow readily follows from this lemma. Indeed, for any $n \in \NN$ and $t \in [0,\infty)$, it holds that $\sH(\bar \rho^n_t \| \gamma) \leq \sH(\rho_0\|\gamma)$ from the definition of the JKO scheme. Since $\KL$ is weakly lower semicontinuous, we further have $\sH(\rho_t\|\gamma) \leq \sH(\rho_0\|\gamma)$ for all $t \in [0,\infty)$. 
By \cref{lem:spectral-control}, we conclude
\begin{equation}
k_{\rho_0} \le \inf_{t \in [0,\infty)} \lambda_{\min}(\Sigma_{\rho_t}) \le \sup_{t \in [0,\infty)} \lambda_{\max}(\Sigma_{\rho_t}) \le 4\sH(\rho_0 \| \gamma) + 2d,
\label{eq: lower bound}
\end{equation}
where 
\begin{equation}
k_{\rho_0} := \exp\!\left(-2\sH(\rho_0 \| \gamma)-d-(d-1)\log(4\sH(\rho_0 \| \gamma)+2d)\right).
\label{eq: lower bound value}
\end{equation}

\section{Main Results}
\label{sec: main}

Theorem~\ref{thm:main} provides a characterization of the IGW gradient flow for $\sH(\cdot \| \gamma)$ through the PIDE (\ref{eq:PIDE}). However, the implicit nature of the PIDE makes it difficult to interpret the flow dynamics. To address that, in this section, we first derive a closed-form expression for the vector field $v_t$, which gives rise to a nonlinear FPE for $\rho_t$. The resulting PDE is nonlinear as the coefficients depend on the marginal law $\rho_t$ through the second moment matrix $\Sigma_{\rho_t}$. Somewhat surprisingly, one can decouple the dynamics of $\Sigma_{\rho_t}$ from $\rho_t$, which results in a linear FPE for $\rho_t$. A standard argument then shows that $\rho_t$ corresponds to the time-marginal flow of a linear SDE that resembles the OU process, which provides a simple probabilistic representation of the IGW gradient flow. Finally, we establish exponential contraction for the IGW gradient flow in relative entropy. Throughout this section, we consider the setting of \cref{thm:main}, with all proofs provided in \cref{sec: proof main}.

\subsection{Fokker--Planck equations}

We first evaluate the limiting subdifferential of $\sH(\cdot\|\gamma)$ and the action of the inverse mobility operator on it. This turns the PIDE from \cref{thm:main} into an FPE with explicit, albeit nonlinear, coefficients.

\begin{theorem}[Nonlinear FPE]
\label{thm:nonlin-FP}
The IGW gradient flow $(\rho_t)_{t \in [0,\infty)}$ from \cref{thm:main} solves the nonlinear FPE
\begin{equation}\label{eq:nonlinear-FP}
    \partial_t\mu_t(x) - \nabla\cdot\left(\frac12\Sigma_{\mu_t}^{-1}\nabla\mu_t(x) + \frac14\big(\Sigma_{\mu_t}^{-1}+\Sigma_{\mu_t}^{-2}\big)x\mu_t(x) \right) =0
\end{equation}
in the distributional sense on $\RR^d\times(0,\infty)$. 
\end{theorem}

From Example 11.1.9 in \cite{ambrosio2005gradient}, one may deduce that for a.e. $t \in [0,\infty)$, $\rho_t \in W^{1,1}_{\mathrm{loc}}(\R^d)$ and $\nabla \rho_t=\rho_t(w_t-\id)$ for $w_t = -\cL_{\rho_t}[v_t]$. A closed-form expression for the velocity field, $v_t$, follows by evaluating the actions of $(\cL_{\rho_t}|_{\cI_{\rho_t}})^{-1}$ on $\frac{\nabla \rho_t}{\rho_t}$ and $\id$. This requires some care due to the lack of a priori regularity estimates on $\rho_t$.

The nonlinearity of the PDE \eqref{eq:nonlinear-FP} comes from the second moment matrix $\Sigma_{\rho_t}$,
 which depends on the marginal law $\rho_t$. 
 The next theorem identifies an autonomous ODE for the second moment matrix flow $t \mapsto \Sigma_{\rho_t}$, which turns the nonlinear PDE (\ref{eq:nonlinear-FP}) into a linear FPE. Let $\mathbb{S}^d_{++}$ denote the cone of symmetric positive definite matrices in $\R^{d \times d}$. Recall that a Borel measure $\mu$ on $\R^d$ is called a subprobability measure if $\mu (\R^d) \le 1$.

\begin{theorem}[Linear FPE and second moment ODE]\label{thm:lin-FP}
    The IGW gradient flow $(\rho_t)_{t \in [0,\infty)}$ is the unique subprobability solution to the linear FPE
     \begin{equation}\label{eq:linear-FP}
         \partial_t\mu_t(x)
         - \nabla\cdot\left(\dfrac12 A_t^{-1}\nabla\mu_t(x) +\dfrac14\bigl(A_t^{-1}+A_t^{-2}\bigr)x\mu_t(x) \right)=0, 
     \end{equation}
   with $\mu_0 = \rho_0$,  where $A:[0,\infty) \to \R^{d \times d}$ is the unique solution of the initial-value problem
\begin{equation}\label{eq:IVP}
\begin{aligned}
\dot{A}_t &= \frac12\left(A_t^{-1} - I\right)
\quad \text{for a.e. } t \in (0,\infty), \quad A_{0} = \Sigma_{\rho_0}, \\
A_t &\in \mathbb{S}^d_{++}
\quad \text{for all } t \in (0,\infty).
\end{aligned}
\end{equation}
Moreover, $A_t = \Sigma_{\rho_t}$ for every $t \in [0,\infty).$
\end{theorem}

The proof of \cref{thm:lin-FP} begins by testing the nonlinear PDE \eqref{eq:nonlinear-FP} against suitable approximations of the quadratic functions $x_i x_j$ to establish that the mapping $t \mapsto\Sigma_{\rho_t}$ is absolutely continuous and identify the matrix ODE. 
Next, we establish uniqueness for the second moment ODE by recasting it as a gradient flow, and then
invoke the uniqueness theory for FPEs developed in Chapter 9 of \cite{bogachev2022fokker} to conclude that (\ref{eq:linear-FP}) admits a unique subprobability solution.

\begin{remark}[Gradient flow interpretation of second moment ODE and Wasserstein comparisons]\label{rem:grad_flow}
We observe that the second moment ODE (\ref{eq:IVP}) admits a natural gradient flow interpretation.
Indeed, when $\KL$ is restricted to centered Gaussian measures and parametrized by the covariance matrix, one has
\[ \sH(\gamma_A\|\gamma) = \frac12\big(\tr(A)-\log\det(A)-d\,\big).\]
Taking the Euclidean gradient  w.r.t. the
Frobenius inner product gives
\[\nabla \sH(\gamma_A\|\gamma) =\frac12(I-A^{-1}).\]
Consequently, (\ref{eq:IVP}) can be written as
\begin{equation}\label{eq:KL_flow}
    \dot A_t = \frac12(A_t^{-1}-I) = -\nabla \sH(\gamma_{A_t}\|\gamma),
    \quad
    A_0=\Sigma_{\rho_0}.
\end{equation}
Thus, the second moment ODE is precisely the Euclidean gradient flow of
$\KL$ restricted to centered Gaussian distributions under the covariance parametrization.

On the other hand, recall that the second moment dynamics for the 2-Wasserstein gradient flow of $\KL$ restricted to the Bures--Wasserstein manifold \cite[Equation (4)]{lambert2022variational}~is 
\begin{equation}\label{eq:W_2-cov}
    \dot B_t = -2(B_t - I).
\end{equation}
\cref{eq:W_2-cov} can be viewed as a `preconditioned' analogue of \eqref{eq:IVP} by writing
\begin{equation}\label{eq:W_2-cov-2}
    \dot B_t = -2(B_t - I) = 4B_t \frac{1}{2}(B^{-1}_t - I) = 4B_t \left(-\nabla \sH(\gamma_{B_t}\|\gamma)\right).
\end{equation}
This suggests that the $\W_2$ gradient flow for $\KL$ converges faster than its $\IGW$ analogue when both flows are `close' to equilibrium since $\|\dot B_t\|_{\F} \approx 4 \|\dot A_t\|_{\F},$ where $\|\cdot\|_{\F}$ denotes the Frobenius norm on $\R^{d \times d}.$ Indeed, consider $I + \varepsilon Z$ for some $Z \in \mathbb S^d_{++}.$ Then $-2\left(I + \varepsilon Z - I\right) = -2\varepsilon Z$, while $\frac{1}{2}\big((I + \varepsilon Z)^{-1} - I\big) = \frac12\left(-\varepsilon Z + o(\varepsilon)\right)$.
\end{remark}

\subsection{SDE representation}

Given the linear FPE \eqref{eq:linear-FP}, we now show that the IGW gradient flow $(\rho_t)_{t \in [0,\infty)}$ agrees with the time-marginal flow of a certain linear SDE.

\begin{theorem}[Linear SDE representation]\label{thm:lin-SDE}
Let $(A_t)_{t \in [0,\infty)}$ be the solution of the matrix ODE \eqref{eq:IVP}. The IGW gradient flow $(\rho_t)_{t \in [0,\infty)}$ agrees with the unique time-marginal flow of the linear SDE
\begin{equation}\label{eq:linear-SDE}
    dX_t = -\frac14\big(A_t^{-1}+A_t^{-2}\big)X_t\,dt + A_t^{-1/2}\,dW_t, \quad
    X_0\sim \rho_0,
\end{equation}
where $(W_t)_{t \in [0,\infty)}$ is a standard Brownian motion in $\R^d$ independent of $X_0$. That is, $\rho_t=\mathrm{Law}(X_t)$ for all $t \in [0,\infty)$. 
\end{theorem}

The linear SDE directly yields a closed-form formula for $\rho_t$, which we present next, together with a few other results that immediately follow from that formula.

\begin{corollary}[Closed-form expression]\label{cor:closed-form}
    Let $(X_t)_{t \in [0,\infty)}$ denote the unique strong solution to the SDE \eqref{eq:linear-SDE}. Then, for all $t \in [0,\infty)$,
    \[
    X_t = \Phi_t\left(X_0 + \int^t_{0} \Phi_s^{-1}A_s^{-1/2}d W_s \right),
    \]
    where $\Phi_t$ is nonsingular and the unique solution to the matrix ODE
    \[ 
    \dot \Phi_t = -\frac14(A_t^{-2}+A_t^{-1})\Phi_t,
    \qquad
    \Phi_0=I.
    \]
    Consequently, $\rho_t =  \big((\Phi_{t})_{\#}\rho_0\big) *\gamma_{Q_t}$
where $*$ denotes convolution and
\begin{equation}\label{eq:Q_t}
    Q_t =\int_0^t \Phi_t\Phi_s^{-1}A_s^{-1}\left(\Phi_s^{-1}\right)^{\intercal}\Phi_t^\intercal\, ds.
\end{equation}
\end{corollary}

\begin{corollary}[Gaussianity]\label{cor:gauss}
     If $\rho_0$ is Gaussian, then $\rho_t$ is Gaussian for every
    $t \geq  0$ with mean vector $m_t$ and second moment matrix $\Sigma_t$ uniquely determined by the system of ODEs,
   \[
        \begin{cases}
        \dot m_{t} = -\frac{1}{4}\left(\Sigma^{-1}_{t} + \Sigma^{-2}_{t}\right)m_{t}, \\
        \dot \Sigma_{t}\, = \frac{1}{2}(\Sigma^{-1}_{t} - I),
        \end{cases}
        \]
        with $m_0 = m_{\rho_0}$ and $\Sigma_0 = \Sigma_{\rho_0}$.
\end{corollary}

\begin{corollary}[Regularity]\label{cor:regularity}
For every $t>0$, the covariance matrix $Q_t$ in \eqref{eq:Q_t} is positive definite.
    Hence $\rho_t$ admits a smooth density w.r.t. Lebesgue measure on $\R^d$.
    Moreover, identifying $\rho_t$ with this density, we have that $(t,x)\mapsto \rho_t(x)$ 
    belongs to $C^{\infty}((0,\infty)\times\R^d).$
\end{corollary}

\begin{remark}[OU comparison]
Although the SDE \eqref{eq:linear-SDE} resembles the OU process, it is not, in general, a deterministic time change of the latter. To see this, consider the one-dimensional case with an admissible initial distribution $\rho_0$ having both mean and variance equal to $1.$ For the OU process,
\[
dY_t=-\frac12 Y_t\,dt+dW_t,
    \qquad Y_0\sim \rho_0,
    \]
    and the variance, $\sigma_t,$ satisfies $\sigma_t \equiv 1$ for all $t.$ For the process \eqref{eq:linear-SDE}, it follows from \eqref{eq:IVP} that the variance $s_t$ evolves according to 
    \[\dot s_t = \frac{1}{s_t + m^2_t}  -\frac12\left(\frac{1}{s_t + m^2_t}  + \frac{1}{(s_t + m^2_t)^2}\right)s_t,\]
    where $m_t$ denotes the mean at time $t.$ Now define
\[F(s,m) \coloneqq \frac{1}{s+m^2} -
    \frac12 \left( \frac{1}{s+m^2} + \frac{1}{(s+m^2)^2}\right)s.\]
Since \(s_0=m_0=1\), we have $F(s_0,m_0)=F(1,1)=\frac18>0$. By the continuity of $F$ and of $t\mapsto(s_t,m_t)$, it follows that
\(F(s_t,m_t)>0\) for all sufficiently small $t > 0$. Consequently,
\[
    s_t = s_0 + \int_0^t F(s_r,m_r)\, dr > 1
\]
for all sufficiently small \(t>0\). As such, for the solution $(X_t)_{t \in [0,\infty)}$ to the SDE \eqref{eq:linear-SDE}, there exists $t^\star > 0$ such that $\mathrm{Law}(X_{t^\star}) \neq \mathrm{Law}(Y_s)$, for all $s \geq 0$. Hence, the process \eqref{eq:linear-SDE} cannot, in general, be obtained as a deterministic time change of the OU process.
\end{remark}

\subsection{Exponential convergence}

Recall from \cref{rem:grad_flow} that the matrix ODE \eqref{eq:IVP} can be interpreted as a (Euclidean) gradient flow for the functional $\Psi(A) =
\frac{1}{2}\left(\tr(A)-\log\det(A)\right)$, which has the unique minimizer $I$. Classical theory for Euclidean gradient flows (cf. \cite{santambrogio2017euclidean}) yields that $A_t$ converges exponentially fast to $I$. Given that, a standard synchronous coupling argument, which couples the SDE \eqref{eq:linear-SDE} with the OU process, then gives exponential convergence of $\rho_t$ towards $\gamma$ in $\W_2$.

The next theorem establishes exponential convergence of $\rho_t$ in relative entropy, which is stronger than $\W_2$-convergence, in view of Talagrand's $T_2$ inequality \cite[Theorem 1.1]{talagrand1996transportation}.

\begin{theorem}[Exponential convergence in relative entropy]
\label{thm:exp_conv}
    The IGW gradient flow $(\rho_t)_{t \in [0,\infty)}$ converges to $\gamma$ exponentially fast in relative entropy, i.e., for all $t \in [0,\infty)$,
    \[
 \sH(\rho_t\|\gamma)
    \le
    \sH(\rho_0\|\gamma)
    e^{-\frac{t}{2\max\{\lambda_{\max}(\Sigma_{\rho_0}),1\}}}.
    \]
\end{theorem}

The proof hinges on Proposition 6.3 of \cite{ambrosio2007gradient}. Consider the relative Fisher information 
\begin{equation}
\sI(\rho \| \gamma) \coloneqq \int_{\R^d} \left \| \nabla \log \left ( \frac{d\rho}{d\gamma} \right )\right \|^2  \, d\rho.
\label{eq: fisher}
\end{equation}
First, we show that the kinetic energy, $\int^T_0 \|v_t\|^2_{L^2(\rho_t)} dt$, and the integrated relative Fisher information, $\int^T_0 \sI(\rho_t\|\gamma) dt$, are finite on any bounded interval. This implies that $t \mapsto \sH(\rho_t\|\gamma)$ is absolutely continuous. Then, the closed-form expression for the velocity field $v_t$ obtained in \cref{thm:nonlin-FP} is combined with the Gaussian log-Sobolev inequality, and the desired convergence follows by applying Gr\"onwall's lemma.

\section{Proofs for \cref{sec: construction}}
 Throughout, we shall identify $\pi$ with its $(y,z)$-marginal when the integrand is independent of $x.$
\label{sec: proof construction}
\subsection{Proof of \cref{lem:local_convexity}}\label{sec:pf:local_conv}
Observe that
\[
\sH(\rho \| \gamma) = \int \rho(x) \log \rho (x) \, dx + \frac{1}{2}\int \| x \|^2 \, d\rho(x) + \frac{d}{2}\log (2\pi).
\]
Section 10.1 in \cite{zhang2024gradient} shows that 
the entropy functional, $\rho \mapsto \int \rho(x) \log \rho (x) \, dx$, is convex along modified generalized geodesics. So, to show that $\KL$ is locally convex,
it suffices to establish the local convexity of
\begin{equation}\label{eq:variance-functional}
\sV(\rho)\coloneqq \int \|x\|^2\, d\rho(x).
\end{equation}
Recall $\nu_t  = (T_t)_{\#}\mu_0.$ We may expand
\begin{equation}\label{eq:At-expansion}
\big((1-t)A_1^{-1} + t A_2^{-1}\big)\big((1-t)A_1 y+tA_2 z\big)
= y + t(A_1^{-1}A_2 z + A_2^{-1}A_1 y - 2y) +O(t^2).\notag
\end{equation}
For simplicity, denote $B\coloneqq A_1^{-1}A_2.$ We now obtain
\begin{align}
\sV(\nu_t)
&=\int \|x\|^2 \, d(T_t)_{\#}\mu_0 \notag \\
&= \int \left\| y + t\bigl(B z + B^{-1}y - 2y\bigr) + O(t^2) \right\|^2
\, d\pi(x,y,z) \notag \\
&= \int \left\| y + t\bigl(B z+ B^{-1}y - 2y\bigr) + O(t^2) \right\|^2
\, d\pi(y,z) \notag \\
&= \int \Bigl\| y + t(z-y) - t(z-y)
+ t\bigl(B z+ B^{-1}y - 2y\bigr) + O(t^2) \Bigr\|^2 \, d\pi(y,z) \notag \\
&= \int \Bigl\| t z + (1-t)y
- t(z-y)
+ t\bigl(B z+ B^{-1}y - 2y\bigr)
+ O(t^2) \Bigr\|^2 \, d\pi(y,z).
\label{eq:add-subtract}
\end{align}
Applying the convexity identity to the squared norm in \eqref{eq:add-subtract}, we get
\begin{align}
\sV(\nu_t)
&= (1-t)\int
\Bigl\| y - t(z-y)
+ t\bigl(B z+ B^{-1}y - 2y\bigr)
+ O(t^2) \Bigr\|^2 \, d\pi(y,z) \notag\\
&\quad
+ t\int
\Bigl\| z - t(z-y)
+ t\bigl(B z+ B^{-1}y - 2y\bigr)
+ O(t^2) \Bigr\|^2 \, d\pi(y,z) \notag\\
&\quad
- t(1-t)\int \|y - z\|^2 \, d\pi(y,z). \notag
\end{align}
Therefore,
\begin{equation}
\label{eq:I1-I2-decomposition}
\begin{split}
\sV(\nu_t)
&\leq \underbrace{(1-t)\int
\Bigl\| y - t(z-y)
+ t\bigl(B z+ B^{-1}y - 2y\bigr)
+ O(t^2) \Bigr\|^2 \, d\pi(y,z)}_{=: I_1} \\
&\quad
+ \underbrace{t\int
\Bigl\| z - t(z-y)
+ t\bigl(B z+ B^{-1}y - 2y\bigr)
+ O(t^2) \Bigr\|^2 \, d\pi(y,z)}_{=: I_2}
+ O(t^2). 
\end{split}
\end{equation}

We next estimate $I_1$ and $I_2$. Rearranging the terms in $I_1$ and expanding the squared norm gives
\begin{align}
I_1
&= (1-t)\int
\Bigl\| y
+ t\left((B - I) z+ (B^{-1} - I)y \right)
+ O(t^2) \Bigr\|^2 \, d\pi(y,z) \notag\\
&= (1-t)\sV(\mu_1)
 + 2(1-t)\int \left\langle
 y,\; t\bigl((B-I)z + (B^{-1}-I)y\bigr)
 \right\rangle d\pi(y,z)
 + O(t^2) \notag \\
&= (1-t)\sV(\mu_1)
 + 2\int \left\langle
 y,\; t\bigl((B-I)z + (B^{-1}-I)y\bigr)
 \right\rangle d\pi(y,z)
 + O(t^2). 
 \label{eq:I1-bound}
\end{align}
Likewise, we have
\begin{equation}\label{eq:I2-bound}
I_2 \le t\,\sV(\mu_2) + O(t^2).
\end{equation}
Substituting \eqref{eq:I1-bound} and \eqref{eq:I2-bound} into
\eqref{eq:I1-I2-decomposition}, we have
\begin{equation}
\label{eq:variance-before-cross-bound}
\sV(\nu_t) 
\leq (1-t)\sV(\mu_1) + t\,\sV(\mu_2)
 + 2\int \left\langle
 y,\; t\bigl((B-I)z + (B^{-1}-I)y\bigr)
 \right\rangle d\pi(y,z) + O(t^2). 
\end{equation}

It remains to bound the last integral in \eqref{eq:variance-before-cross-bound}. We first rewrite it as
\begin{equation}
\label{eq:cross-pre-substitution}
\begin{split}
&2\int \left\langle y,\; t\bigl((B-I)z + (B^{-1}-I)y\bigr) \right\rangle d\pi(y,z) \\
&= 2\int \left\langle y,\; t\bigl((B-I)z + (B-I)y - (B-I)y + (B^{-1}-I)y\bigr) \right\rangle d\pi(y,z) \\
&= 2t\int \left\langle y,\; (B-I)(z-y) + (B + B^{-1} - 2I)y \right\rangle d\pi(y,z) \\
&= 2t\int \left\langle y,\; (B-I)(z-y) + B^{-1}(B-I)^2 y \right\rangle d\pi(y,z) \\
&= 2t\left(\int \langle y,(B-I)(z-y)\rangle d\pi(y,z)
   + \int \langle y, B^{-1}(B-I)^2 y\rangle d\pi(y,z)\right)\\
&\le 2t\left(\int \|y\|\,\|B-I\|_{\text{op}}\,\|z-y\|\, d\pi(y,z)
+ \int \|B^{-1}\|_{\text{op}}\,\|B-I\|_{\text{op}}^2\,\|y\|^2\, d\mu_1(y)\right) \\
&\le 2t\Bigg(\|B-I\|_{\text{op}}\,M_2^{1/2}(\mu_1)\left(\int \|z-y\|^2 d\pi(y,z)\right)^{1/2} + \|B^{-1}\|_{\text{op}}\,\|B-I\|_{\text{op}}^2 M_2(\mu_1)\Bigg). 
\end{split}
\end{equation}
Denoting by $\|\cdot\|_{\op}$ the operator norm on $\R^d,$ we recall the following bounds from Section 10.1 of \cite{zhang2024gradient}:
\begin{align}
\|B^{-1}\|_{\mathrm{op}}
&\leq c_{2,\pi}^2, \label{eq:B-op-bound}\\
\|B - I\|_{\text{op}}
&\leq \frac{c_{2,\pi}}{2}
    \left(M_2(\mu_0)
    \int \|z - y\|^2 \, d\pi(y,z)\right)^{1/2}
, \label{eq:B-minus-I-bound}\\
\int \|z - y\|^2 \, d\pi(y,z)
&\leq \frac{2\sqrt{2}}{c_{1,\mu_0}}
\left(2 \int\left|\langle y, y' \rangle -\langle z, z' \rangle\right|^2 \, d\pi \otimes \pi+ 3 \,\IGW^2(\mu_1,\mu_0)
\right). \label{eq:zy-bound}
\end{align}
Substituting \eqref{eq:B-op-bound} and \eqref{eq:B-minus-I-bound} into
\eqref{eq:cross-pre-substitution}, we obtain
\begin{align}
&2 \int
\left\langle
    y,\;
    t\bigl((B - I)z + (B^{-1} - I)y\bigr)
\right\rangle
\, d\pi(y,z) \notag\\
&\leq
2t
\Bigg(
    \frac{c_{2,\pi}}{2}
    M_2^{1/2}(\mu_0)
    M_2^{1/2}(\mu_1)
+\frac{c_{2,\pi}^4}{4}
    M_2(\mu_1)
    M_2(\mu_0)
\Bigg)
\int \|z - y\|^2 \, d\pi(y,z). \label{eq:cross-after-B-substitution}
\end{align}
Substituting \eqref{eq:zy-bound} into \eqref{eq:cross-after-B-substitution}, we get
\begin{align}
&2 \int
\left\langle
    y,\;
    t\bigl((B - I)z + (B^{-1} - I)y\bigr)
\right\rangle
\, d\pi(y,z) \notag\\
&\leq
t\frac{8\sqrt{2}}{c_{1,\mu_0}}
\Big(\frac{c_{2,\pi}}{2}M_2^{1/2}(\mu_0) M_2^{1/2}(\mu_1) + \frac{c_{2,\pi}^4}{4}M_2(\mu_1)M_2(\mu_0)\Big) \notag\\
&\qquad \times
\Bigg(
    \int
    \left| \langle y, y' \rangle - \langle z, z' \rangle \right|^2
    \, d\pi \otimes \pi + 3\, \mathsf{IGW}^2(\mu_1,\mu_0)
\Bigg). \label{eq:cross-final-bound}
\end{align}
Putting \eqref{eq:variance-before-cross-bound} and \eqref{eq:cross-final-bound} together,
and using the definition of $\beta_{\pi}$, we obtain
\begin{align*}
\sV(\nu_t)
&\le (1-t)\sV(\mu_1)+t\sV(\mu_2)
  + t\,\frac{8\sqrt{2}}{c_{1,\mu_0}}\,\beta_{\pi}\,
     \int \bigl|\langle y,y'\rangle-\langle z,z'\rangle\bigr|^2\, d\pi \otimes \pi \\
&\qquad
  + t\,\frac{12\sqrt{2}}{c_{1,\mu_0}}\,\beta_{\pi}\, \IGW^2(\mu_1,\mu_0)
  + O(t^2).
\end{align*}
This proves the desired local convexity. \qed

\subsection{Proof of \cref{lem:spectral-control}}\label{sec:pf:spec_control}
    Recall that $\Sigma_\rho$ denotes the second moment matrix of $\rho$. Let $m_\rho$ and $\tilde{\Sigma}_\rho$ denote the mean and covariance matrix of $\rho$, respectively. Observe that, by Weyl's inequality, $\lambda_{\min}(\Sigma_{\rho}) \geq \lambda_{\min}(\tilde{\Sigma}_\rho)$ and $\lambda_{\max}(\tilde{\Sigma}_\rho) \leq \lambda_{\max}(\Sigma_{\rho})$ (see Corollary 3.2.3 from \cite{bhatia2013matrix}).
    
    We first establish an upper bound on the largest eigenvalue of $\Sigma_\rho$. Recall Talagrand's $T_2$ transport inequality for the standard Gaussian $\gamma = \mathcal{N}(0, I)$ (e.g., Theorem 1.1 in \cite{talagrand1996transportation}),
    \[ 
    \sW_2(\rho, \gamma) \leq \sqrt{2\sH(\rho\|\gamma)} \leq \sqrt{2C}. 
    \]
    Using the  inequality $\|x\|^2 \leq 2\|x-y\|^2 + 2\|y\|^2$ integrated over the optimal transport plan between $\rho$ and $\gamma$, we obtain
    \begin{align*}
        \lambda_{\max}(\Sigma_\rho) &\leq \tr(\Sigma_\rho) = \int_{\mathbb{R}^d} \|x\|^2 d\rho(x) \leq 2\sW^2_2(\rho, \gamma) + 2\int_{\mathbb{R}^d} \|y\|^2 d\gamma(y) \\
        &\leq 2(\sqrt{2C})^2 + 2d = 4C + 2d =: M.
    \end{align*}

    Next, recall that among all distributions with mean $m_\rho$ and covariance $\tilde{\Sigma}_\rho$, the entropy functional $\rho \mapsto \int \rho(x) \log \rho (x) \, dx$ is minimized at $\rho = \cN(m_\rho,\tilde{\Sigma}_\rho)$. This implies that
    \begin{align*}
        C \geq \sH(\rho\|\gamma) &\geq \sH(\mathcal{N}(m_{\rho}, \tilde{\Sigma}_\rho)\|\gamma) \\
        &= \frac{1}{2}\Bigl(\tr(\tilde{\Sigma}_\rho) + \|m_{\rho}\|^2 - d - \log \det(\tilde{\Sigma}_\rho)\Bigr) \\
        &\geq \frac{1}{2}\Bigl( - d - \log \det(\tilde{\Sigma}_\rho)\Bigr)  \\
        &\geq \frac{1}{2}\Bigl(-d - \log \lambda_{\min}(\tilde{\Sigma}_\rho) - (d - 1)\log M\Bigr),
    \end{align*}
    where the last inequality follows because $\log \det(\tilde{\Sigma}_\rho) \le \log \lambda_{\min}(\tilde{\Sigma}_\rho) + \sum_{i=2}^d \log \lambda_{\max}(\tilde{\Sigma}_\rho) \leq \log \lambda_{\min}(\tilde{\Sigma}_\rho) + (d-1)\log M$. 
    Rearranging the inequality to isolate the smallest eigenvalue, we get
    \[ 
    \log \lambda_{\min}(\tilde{\Sigma}_\rho) \geq -2C - d - (d - 1)\log M. 
    \]
    Exponentiating and substituting $M = 4C + 2d$ yields
    \[ 
    \lambda_{\min}(\tilde{\Sigma}_\rho) \geq \exp\Big(-2C - d - (d - 1)\log(4C + 2d)\Big). 
    \]
    The result follows from the initial comparison $\lambda_{\min}(\Sigma_{\rho}) \geq \lambda_{\min}(\tilde{\Sigma}_\rho)$.
\qed
\subsection{Proof of \cref{thm:main}}\label{sec:construction_proof}
\subsubsection{Proof of \cref{thm:main}~(i)}
Since $\KL$ is rotationally invariant, lower bounded, weakly lower semicontinuous (see Proposition 4.9 in \cite{polyanskiy2025information}), and regular in the sense of  \cite[Definition 10.1.4]{ambrosio2005gradient}, the conclusion of \cref{thm:zz_main1}~(i) holds with $\sF = \sH(\cdot \| \gamma)$. It remains to establish $v_t \in - (\cL_{\rho_t}|_{\cI_{\rho_t}})^{-1}[\partial_\ell\sH(\rho_t\|\gamma)]$.

 In view of the proof of Proposition 4.7 in \cite{zhang2024gradient}, we need to establish an analogue of Proposition 4.6 in \cite{zhang2024gradient} and strengthen the convergence of $\bar \rho^n \to \rho$ from pointwise weak convergence to uniform $\sW_2$-convergence. 

 The proof of Proposition 4.6 in \cite{zhang2024gradient} relies on the preceding Proposition 4.5 concerning a uniform Cauchy estimate for $\bar{\rho}^n$ in $\IGW$, which in turn relies on $\lambda$-convexity of the functional along (modified) generalized geodesics. Due to the lack of $\lambda$-convexity for $\sH(\cdot \| \gamma)$, we need to separately establish such a Cauchy estimate for $\bar \rho^n$, which requires substantial work.
Specifically, for $n,m\in\NN$ with $\tau=\frac{\delta}{n}$ and $\eta=\frac{\delta}{m}$, we consider the piecewise-constant
IGW--JKO interpolants
$\bar\rho^n$ and $\bar\rho^m$ and  will establish below the following uniform Cauchy estimate,
\begin{equation}
\label{eq: cauchy}
\IGW\big(\bar{\rho}^n_t,\bar{\rho}^m_t\big) \lesssim_{\rho_0} \sqrt{\tau} + \sqrt{\eta}
\end{equation}
for sufficiently small $\tau, \eta$. 
 Given the Cauchy estimate \eqref{eq: cauchy}, one may follow the proofs of Propositions 4.6 and 4.7 in \cite{zhang2024gradient} verbatim to obtain $\bar \rho^n \to \rho$ uniformly in $\sW_2$ and conclude $-\cL_{\rho_t} [v_t] \in \partial_{\ell} \sH(\rho_t\|\gamma)$, which leads to the desired claim.\footnote{Propositions 4.6 and 4.7 in \cite{zhang2024gradient} assume $\lambda$-convexity, but the assumption is used only to invoke the uniform Cauchy estimate.}

To establish the estimate \eqref{eq: cauchy}, we introduce a cross-partition error function. Fix $n\in\NN$, and define the piecewise linear function $\ell_\tau:[0,\delta]\to\RR$ by $\ell_\tau(t)\coloneqq \frac{t-(i-1)\tau}{\tau}$, for $i=1,\ldots,n,\ t\in ((i-1)\tau,i\tau],$ and $\ell_\tau(0) \coloneqq 0$. For $\nu\in\cP_2(\RR^d)$, set 
\[
d_\tau(t;\nu)\coloneqq \left((1-\ell_\tau(t))\IGW^2(\rho^\tau_{i-1},\nu) + \ell_\tau(t) \IGW^2(\rho^\tau_i,\nu)\right)^{1/2}, \ \  i=1,\ldots,n,\ t\in((i-1)\tau,i\tau].
\]

\begin{definition}[Cross-partition error function]
    For $n,m\in\NN$ with $\tau=\frac{\delta}{n}$ and $\eta=\frac{\delta}{m}$, define the function $d_{\tau\eta}:[0,\delta]^2\to\RR$ by 
    \begin{equation}
    d_{\tau\eta}(t,s)\coloneqq \left((1-\ell_\eta(s))d^2_\tau(t;\rho^\eta_{j-1}) + \ell_\eta(s)d^2_\tau(t;\rho^\eta_{j})\right)^{1/2},\label{eq:error_func}
    \end{equation}
    for $(t,s)\in((i-1)\tau,i\tau]\times ((j-1)\eta,j\eta]$ and $(i,j)\in\{1,\ldots,n\}\times \{1,\ldots,m\},$ and define the function on the boundary of $[0,\delta]^2$ by replacing the corresponding discrete interpolation by the common initial distribution $\rho_0$.
\end{definition}

By definition,
\begin{align*}
    d^2_{\tau\eta}(t,t) =&(1-\ell_\eta(t))(1-\ell_\tau(t))\IGW^2(\rho^\tau_{i-1},\rho^\eta_{j-1}) + (1-\ell_\eta(t))\ell_\tau(t) \IGW^2(\rho^\tau_i,\rho^\eta_{j-1})\\
    &\quad\quad+  \ell_\eta(t)(1-\ell_\tau(t))\IGW^2(\rho^\tau_{i-1},\rho^\eta_{j}) + \ell_\eta(t)\ell_\tau(t) \IGW^2(\rho^\tau_i,\rho^\eta_{j}),
\end{align*}
for $t\in ((i-1)\tau,i\tau]\cap ((j-1)\eta,j\eta]$, whereas $\IGW(\bar{\rho}^n_t,\bar{\rho}^m_t) = \IGW(\rho^\tau_i,\rho^\eta_{j})$. As such, $d^2_{\tau\eta}(t,t)$ is a convex combination of $\IGW^2$ between $\rho^\tau_{i-1},\rho^\tau_i$ and $\rho^\eta_{j-1},\rho^\eta_{j}$, so that
\begin{equation}
\label{eq: cauchy 1}
\begin{split}
\big|\IGW\big(\bar{\rho}^n_t,\bar{\rho}^m_t\big) - d_{\tau\eta}(t,t)\big|&\lesssim_{\rho_0}  \IGW(\rho^\tau_{i-1},\rho^\tau_i) + \IGW(\rho^\eta_{j-1},\rho^\eta_{j})\\
    & \lesssim_{\rho_0}\sqrt{\tau} + \sqrt{\eta} 
    \end{split}
\end{equation}
where the second inequality follows by Proposition 4.1 in \cite{zhang2024gradient}. It remains to establish a uniform upper bound on $d_{\tau\eta}(t,t)$. What follows is to find an upper bound on the time derivative of $d^2_{\tau\eta}(t,t)$ from the variational inequality in \cref{lem:prox_ineq} below, which, combined with the Gr\"onwall lemma, leads to 
\begin{equation}
\label{eq: cauchy 2}
d_{\tau \eta}(t,t) \lesssim_{\rho_0} \sqrt{\tau} + \sqrt{\eta}
\end{equation}
for sufficiently small $\tau, \eta$. 
The desired Cauchy estimate \eqref{eq: cauchy} follows by combining \eqref{eq: cauchy 1} and \eqref{eq: cauchy 2}.
 The proof of \cref{lem:prox_ineq}, which is deferred to \cref{app:2}, relies on the local convexity of $\sH(\cdot \| \gamma)$, as established in \cref{lem:local_convexity}, combined with that of $\IGW^2$, which appears in  Section 10.1 of \cite{zhang2024gradient} and is restated in \cref{lem:IGW_conv} below.

 \begin{lemma}\label{lem:prox_ineq}
    For any $\mu_1 \in \argmin \sH(\mu \| \gamma) + \frac{1}{2  \tau} \IGW^2(\mu_0,\mu),$ if there exists a modified generalized geodesic (see \cref{def:MGG}), between $\mu_1, \mu_2$ w.r.t. $\mu_0,$ then we have
\begin{align*}
&\sH(\mu_1 \| \gamma) + \frac{1}{2\tau} \, \IGW^2(\mu_1,\mu_0) \\
&\le
\sH(\mu_2 \| \gamma)
- \frac{1}{2\tau}\Bigg(
\left(1 - \frac{8\sqrt{2}}{c_{1,\mu_0}} \, \IGW(\mu_0,\mu_1)\, \alpha_\pi
- 2\tau \frac{8\sqrt{2}}{c_{1,\mu_0}}\,\beta_\pi
\right)\times \\  &\quad\int_{(\R^d \times \R^d)^2} \bigl|\langle y,y'\rangle-\langle z,z'\rangle\bigr|^2\, d\pi_{1,2}\otimes\pi_{1,2}(y,z,y',z')
- 2\tau \frac{12\sqrt{2}}{c_{1,\mu_0}}\,\beta_\pi \,
\IGW^2(\mu_0,\mu_1)\\
&\quad- \IGW^2(\mu_2,\mu_0)
- \frac{12\sqrt{2}}{c_{1,\mu_0}}\,\alpha_\pi \,\IGW^3(\mu_0,\mu_1)\Bigg),
\end{align*}
where $\alpha_\pi = 1+ \frac{c_{2,\pi}^4}{2} M_2(\mu_1)M_2(\mu_0)+ c_{2,\pi} M_2(\mu_1)$, and $c_{1,\mu_0}, c_{2,\pi}$ and $\beta_\pi$ are defined as in \cref{lem:local_convexity}.
\end{lemma}

Given \cref{lem:prox_ineq}, the proof of \eqref{eq: cauchy 2} is similar to that of Proposition 4.5 in \cite{zhang2024gradient}. For completeness, we provide the full proof in \cref{app: 3}. 
\qed

\subsubsection{Proof of \cref{thm:main}~(ii)}
Set $k_{\rho_0}$ as in \eqref{eq: lower bound value}. Part (i) of the theorem constructs a gradient flow $\rho^0: [0,\delta_0] \to \cP_2(\R^d)$, starting at $\rho_0$,  with $\delta_0 = \frac{(1-1/\sqrt{2})k_{\rho_0}}{2^{9/4}\sH(\rho_0\|\gamma)}$, which satisfies \eqref{eq:PIDE} with $(0,\infty)$ replaced by $(0,\delta_0)$ for some velocity field $(v_t^0)_{t \in [0,\delta_0]}$. Observe that $\lambda_{\min}(\Sigma_{\rho^0_{\delta_0}}) \ge k_{\rho_0}$ (cf. discussion following \cref{lem:spectral-control}). 

Next, apply Part (i) of the theorem to construct another gradient flow $\rho^1: [\delta_0, 2\delta_0] \to \cP_2(\R^d)$, starting at $\rho^1_{\delta_0} = \rho^0_{\delta_0}$, which satisfies \eqref{eq:PIDE} with $(0,\infty)$ replaced by $(\delta_0,2\delta_0)$ for some velocity field $(v_t^1)_{t \in [\delta_0,2\delta_0]}$. We concatenate $(\rho^0_t,v^0_t)_{t \in [0,\delta_0]}$ and $(\rho^1_t,v^1_t)_{t \in [\delta_0,2\delta_0]}$ as
    \[
        (\hat{\rho}_t, \hat{v}_t) = 
        \begin{cases} 
            (\rho^0_t, v^0_t) & \text{for } t \in [0,\delta_0], \\
            (\rho^1_t, v^1_t) & \text{for } t \in [\delta_0,2\delta_0].
        \end{cases}
    \]
    Since $\hat{\rho}$ is $\W_2$-continuous on $[0, 2\delta_0 ]$ by construction, we conclude by \cite[Lemma 8.1.2]{ambrosio2005gradient} that $(\hat{\rho}_t, \hat{v}_t)$ satisfies \eqref{eq:PIDE} with $(0,\infty)$ replaced by $(0,2\delta_0)$. Repeating this argument gives the desired result. \qed

\section{Proofs for \cref{sec: main}}
\label{sec: proof main}

Throughout, we shall abbreviate $\cL_{\rho_t}$ by $\cL_t$, $\cI_{\rho_t}$ by $\cI_t$, and $\Sigma_{\rho_t}$ by $\Sigma_t$.
\subsection{Proof of \cref{thm:nonlin-FP}}
\label{pf:non-lin-FP}
By \cref{thm:main}, the gradient flow $(\rho_t)_{t \in [0,\infty)}$ satisfies the PIDE (\ref{eq:PIDE}). 
We convert the PIDE into a PDE by obtaining a closed-form expression for
$\partial_\ell \sH(\rho_t\|\gamma)$, and applying
$(\cL_t|_{\cI_{t}})^{-1}$ to it. Define
\[
    w_t \coloneqq -\cL_t[v_t] \in \partial_\ell \sH(\rho_t\|\gamma).
\]
Since $\KL$ satisfies the conditions of Example 11.1.9 in
\cite{ambrosio2005gradient}, and since
$\partial_\ell \sH(\rho_t\|\gamma)$ is nonempty for a.e. $t \in [0,\infty)$
by construction, we have
\[
    \rho_t \in W^{1,1}_{\text{loc}}(\mathbb{R}^d), \qquad \nabla \rho_t = \rho_t(w_t-\id)
\]
for a.e. $t \in [0,\infty)$ (in what follows, we suppress the qualifier
``for a.e. $t$'' for ease of notation). 
In addition, since $w_t \in L^2(\rho_t;\mathbb{R}^d)$ and
$\rho_t \in \cP_2(\R^d)$ by construction, it follows that $\rho_t \in W^{1,1}(\mathbb{R}^d)$. 
Rearranging, we obtain
\[
    w_t = \frac{\nabla\rho_t}{\rho_t}+\id
    \quad \rho_t\text{-a.s.}
\]

We shall now verify that the actions of $(\cL_t|_{\cI_{t}})^{-1}$ on $\frac{\nabla \rho_t}{\rho_t}$ and $\id$ are well-defined and find the explicit expression for
\begin{equation}
\label{eq:inverse_1}
    v_t(x) 
    = -(\cL_t|_{\cI_{t}})^{-1}\left[\frac{\nabla\rho_t}{\rho_t}\right](x) - (\cL_t|_{\cI_{t}})^{-1}\left[\id\right](x).
\end{equation}

First, since $\rho_t \in \cP_2(\R^d)$ and $\int_{\RR^d}x\id(x)^\intercal \, d\rho_t(x) = \int_{\RR^d}xx^\intercal \, d\rho_t(x)$, which is symmetric, we conclude that $\id \in \cI_{t}$, so that $(\cL_t|_{\cI_{t}})^{-1}\left[\id\right]$ is well-defined; see \cref{prop:fredholm}~(ii). Define the mapping
$
S_t: x \mapsto  \frac14 \Sigma_t^{-1}x,
$
which is in $L^2(\rho_t;\RR^d)$. Since 
\[
\int_{\R^d}x S_t(x)^{\intercal} \, d\rho_t(x) = \frac{1}{4} \left (\int_{\R^d}xx^\intercal \, d\rho_t(x) \right)\Sigma^{-1}_t = \frac{1}{4} I,
\]
we have $S_t \in \cI_{t}$. 
Observe that $\cL_t[S_t](x) = \frac{1}{2}x + \frac{1}{2}x = x$,
which implies that
\begin{equation}\label{eq:inv_2}
    (\cL_t|_{\cI_{t}})^{-1}\left[\id\right](x) = \frac{1}{4} \Sigma^{-1}_t x.
\end{equation}

Second, we shall verify that $(\cL_t|_{\cI_{t}})^{-1}\left[\frac{\nabla\rho_t}{\rho_t}\right](x)$ is well-defined. Again, since $w_t \in L^2(\rho_t;\RR^d)$ and $\rho_t \in \cP_2(\R^d)$ by construction, we see that $x \mapsto \frac{\nabla\rho_t}{\rho_t}(x)$ is in $L^2(\rho_t;\RR^d)$. It remains to verify that 
\[
\int_{\R^d}y \frac{\nabla\rho_t(y)^\intercal}{\rho_t(y)} \,d \rho_t(y) = \int_{\R^d}y \nabla\rho_t(y)^\intercal \,d y = \left( \int_{\RR^d} y_i \partial_j\rho_t(y) \,d y\right)^d_{i,j = 1}
\] 
is integrable and symmetric.

We first verify integrability. Since $\nabla\rho_t=\rho_t(w_t-\id), w_t\in L^2(\rho_t;\mathbb R^d)$, and $\rho_t\in\mathcal P_2(\mathbb R^d)$, we have
\[
    \int_{\mathbb R^d}\|y\|\,\|\nabla\rho_t(y)\|\,d y
    \leq M_2(\rho_t)^{1/2}\|w_t\|_{L^2(\rho_t)} + M_2(\rho_t) <\infty.
\]

Next,
we establish symmetry, which needs some care. If $\rho_t$ were known to be in $C^1$ and satisfy $\lim_{\|y\| \to \infty}\|y\| \rho_t(y) = 0$, then one may apply standard integration by parts to conclude that $\int_{\R^d}  y_i \partial_j \rho_t(y) \, dy = - \int_{\R^d} (\partial_j y_i) \rho_t(y) \, dy = -\delta_{ij}$, which would imply the symmetry. However, such regularity estimates on $\rho_t$ are not available a priori, so a careful argument is needed to justify the integration by parts. 
Let $\eta_R \in C_c^\infty(\mathbb{R}^d)$ be a smooth cut-off function such that
\begin{equation}
    \eta_R \equiv 1 \ \ \text{on} \ \ \mathcal{B}_R, 
    \quad 
    \eta_R \equiv 0  \ \ \text{outside} \ \ \mathcal{B}_{3R/2},
    \quad 
    0 \leq \eta_R \leq 1,  \quad \text{and}
    \quad 
    \|\nabla \eta_R\| \lesssim R^{-1},
    \label{eq: cutoff}
\end{equation}
where $\cB_{r}$ denotes the open ball in $\R^d$ centered at the origin and radius $r > 0$. 
As $R \to \infty$,
\[
    \eta_R(y)y_i\partial_j\rho_t(y)
    \to y_i\partial_j\rho_t(y)
    \quad \text{pointwise}, \text{ and}
    \qquad
    |\eta_R(y)y_i\partial_j\rho_t(y)|
    \leq |y_i\partial_j\rho_t(y)|.
\]
So, by the dominated convergence theorem,
\[
\lim_{R \to \infty} \int_{\RR^d} \eta_R(y)y_i\partial_j\rho_t(y) \, dy  = \int_{\RR^d} y_i\partial_j\rho_t(y) \, dy.
\]
Since $y \mapsto \eta_R(y)y_i$ is in $C^{\infty}_c(\cB_{2R})$ and $\rho_t \in W^{1,1}(\RR^d)$, the weak integration by parts  gives
\begin{align*}
    \int_{\RR^d}\eta_R(y) y_i\partial_j\rho_t(y) \, dy &= \int_{\cB_{2R}}\eta_R(y) y_i\partial_j\rho_t(y) \, dy\\
    &=  - \int_{\cB_{2R}} y_i \partial_j \eta_R(y)\, d\rho_t(y) - \delta_{ij} \int_{\cB_{2R}} \eta_R(y)\, d\rho_t(y).
\end{align*}
For the first term,
\[
\left|\int_{\cB_{2R}} y_i \partial_j \eta_R(y) \rho_t(y) \, dy\right| \lesssim \frac{1}{R} \int_{\RR^d} |y_i|\, d\rho_t(y) \to 0 \text{ as $R \to \infty$} 
\]
since $\rho_t$ has finite first moment.
For the second term, by the dominated convergence theorem, we have
\[ 
\lim_{R \to \infty} \delta_{ij} \int_{\cB_{2R}} \eta_R(y) \rho_t(y) \, dy  =  \delta_{ij} \int_{\RR^d} \rho_t(y)\, dy = \delta_{ij}.
\]
As such, we conclude
\[ 
\int_{\RR^d} y_i\partial_j\rho_t(y) \, dy = \lim_{R \to \infty} \int_{\RR^d} \eta_R(y)y_i\partial_j\rho_t(y) \, dy = -\delta_{ij},
\] 
and hence
\[
\int_{\R^d}y \frac{\nabla\rho_t(y)^\intercal}{\rho_t(y)}\,d \rho_t(y) = \int_{\R^d}y \nabla\rho_t(y)^\intercal\,d y = -I.
\] 

We have verified that $(\cL_t|_{\cI_{t}})^{-1}\left[\frac{\nabla\rho_t}{\rho_t}\right](x)$ is well-defined. By \cref{prop:fredholm}~(ii),  
\[(\cL_t|_{\cI_{t}})^{-1}\left[\frac{\nabla\rho_t}{\rho_t}\right](x)
    = \frac12 \Sigma_t^{-1}\frac{\nabla\rho_t}{\rho_t}(x)
      - \Sigma_t^{-1}
      B_tx,
\]
where $B_t$ is the unique symmetric solution of 
\[
\Sigma_t B_t + B_t \Sigma_t = \frac{1}{2} \int_{\R^d} \frac{\nabla\rho_t(y)}{\rho_t(y)}y^\intercal\,d \rho_t(y) = -\frac{1}{2}I.
\]
Solving this equation gives $B_t = -\frac{1}{4}\Sigma^{-1}_t,$ so that
\begin{equation}\label{eq:inv_3}    (\cL_t|_{\cI_{t}})^{-1}\left[\frac{\nabla\rho_t}{\rho_t}\right](x)
    = \frac12 \Sigma_t^{-1}\frac{\nabla\rho_t}{\rho_t}(x)
      + \frac{1}{4}\Sigma_t^{-2}x.
\end{equation}

Finally, substituting \eqref{eq:inv_2} and \eqref{eq:inv_3} into \eqref{eq:inverse_1} gives
\begin{equation}\label{eq:v_t_closed}
    v_t(x) = - \frac{1}{2}\Sigma^{-1}_t \frac{\nabla\rho_t}{\rho_t}(x)  - \frac{1}{4}\left(\Sigma^{-1}_t + \Sigma^{-2}_t\right)x \qquad \rho_t\text{-a.s.}\quad t \in [0,\infty) \text{ a.e.}
\end{equation}
The desired nonlinear FPE follows by plugging in \eqref{eq:v_t_closed} into the continuity equation in \eqref{eq:PIDE}.
\qed

\subsection{Proof of \cref{thm:lin-FP}}
\label{pf:lin-FP}

We first establish the following lemma concerning the kinetic energy estimate. 

\begin{lemma}\label{lem:kin_en}
    Consider $(v_t)_{t \in [0,\infty)},$ the velocity field of the gradient flow $(\rho_t)_{t \in [0,\infty)}$. For all $T \in [0,\infty)$, we have 
    \[
    \int^T_{0} \int_{\RR^d} \|v_t\|^2 \, d\rho_t dt < \infty.
    \]
\end{lemma}
\begin{proof}
    Recall from \eqref{eq: lower bound} that $\inf_{t \in [0,\infty)}\lambda_{\min} (\Sigma_{t}) \ge k_{\rho_0}$. By the definition of the mobility operator $\cL_{t}$, one has
    \begin{align*}
    g_{\rho_t}(v_t,v_t) := \llangle v_t,  \cL_{t}[v_t]\rrangle_{L^2(\rho_t;\RR^d)}
   \geq 2k_{\rho_0} \|v_t\|^2_{L^2(\rho_t;\RR^d)}.
\end{align*}
To see the last inequality, observe that, as $v_t \in \cI_t$,
\[
\begin{split}
\llangle v_t,  \cL_{t}[v_t]\rrangle_{L^2(\rho_t;\RR^d)} &= 2\left ( \int_{\R^d} v_t^\intercal \Sigma_t v_t \, d\rho_t + \left \|  \int_{\R^d} xv_t(x)^\intercal \, d\rho_t(x) \right \|_{\F}^2 \right ) \\
&\ge 2\lambda_{\min} (\Sigma_t) \| v_t \|_{L^2(\rho_t;\R^d)}^2,
\end{split}
\]
where $\| \cdot \|_{\F}$ denotes the Frobenius norm.
Applying the energy inequality from Section 10.3 in \cite{zhang2024gradient}, $\int^T_0 g_{\rho_t}(v_t,v_t) \,dt \leq 2\sH(\rho_0\|\gamma)$, 
we obtain
\begin{align*}
    \int^T_{0} \int_{\RR^d} \|v_t\|^2 \, d\rho_t dt \le k_{\rho_0}^{-1}\sH(\rho_0 \| \gamma) < \infty.
\end{align*}
\end{proof}

\begin{proof}[Proof of \cref{thm:lin-FP}]
Consider any $\xi \in C_c^\infty((0,T))$ and 
$\eta_R \in C_c^\infty(\R^d)$ satisfying \eqref{eq: cutoff}.
Since $(\rho_t)_{t \in [0,\infty)}$ is a distributional solution of \eqref{eq:nonlinear-FP} on $(0,T) \times \R^d,$ for any $0 < T < \infty,$ we have that for $i,j \in \{ 1,\dots, d \}$, 
\[
\int_{0}^{T}\int_{\mathbb{R}^{d}}
\partial_t \big (\xi(t)\eta_R(x)x_i x_j \big)
\,d\rho_t(x)\,dt
=
-\int_{0}^{T}\int_{\mathbb{R}^{d}}
\left\langle
\nabla_x \big (\xi(t)\eta_R(x)x_i x_j \big),
v_t(x)
\right\rangle
\,d\rho_t(x)\,dt.
\]
Denoting the $i$-th component of $v_t$ by $v_t^{(i)}$,
we have
\begin{multline}
\label{eq:moment-cutoff-identity}
\int_{0}^{T}\int_{\mathbb{R}^{d}}
\xi'(t)\eta_R(x)x_i x_j
\,d\rho_t(x)\,dt
=
-\int_{0}^{T}\int_{\mathbb{R}^{d}}
\xi(t) x_i x_j
\big\langle \nabla \eta_R(x),
v_t(x)
\big\rangle
\,d\rho_t(x)\,dt \\
-\int_{0}^{T}\int_{\mathbb{R}^{d}}
\xi(t)\eta_R(x)
\big(x_j v_t^{(i)} (x) + x_i v_t^{(j)}(x) \big)
\,d\rho_t(x)\,dt.
\end{multline}
Since $M_2(\rho_t) \leq d(4\sH(\rho_0\|\gamma) + 2d)$ (cf. \eqref{eq: lower bound}), it follows that $\int_{0}^{T}\int_{\mathbb{R}^{d}}
|\xi'(t)x_i x_j
\,d\rho_t(x)|\,dt < \infty$, which, by the dominated convergence theorem, implies that
\begin{equation}\label{eq:term_1}
    \lim_{R \to \infty} \int_{0}^{T}\int_{\mathbb{R}^{d}}
\xi'(t)\eta_R(x)x_i x_j
\,d\rho_t(x)\,dt = \int^T _0\xi'(t) \int_{\mathbb{R}^{d}}x_i x_j
\,d\rho_t(x)\,dt.
\end{equation}
Likewise, $\int^T_0 \int_{\R^d} |\xi(t)x_k v^{(k)}_t(x)| \, d\rho_t(x)dt < \infty$ for $k \in \{ 1,\dots,d \}$ by \cref{lem:kin_en}, so that we have
\begin{equation}\label{eq:term_3}
\begin{split}
    &\lim_{R \to \infty}\int_{0}^{T}\int_{\mathbb{R}^{d}}
\xi(t)\eta_R(x)
\big(x_j v_t^{(i)} (x) + x_i v_t^{(j)}(x) \big)
\,d\rho_t(x)\,dt\\
&= \int_{0}^{T}\xi(t)\int_{\mathbb{R}^{d}}
\big(x_j v_t^{(i)} (x) + x_i v_t^{(j)}(x) \big)
\,d\rho_t(x)\,dt.
\end{split}
\end{equation}

We shall show that the first term on the right-hand side of \eqref{eq:moment-cutoff-identity} vanishes as $R \to \infty$. Observe that
\begin{align*}
&\left|
\int_{0}^{T}
-\xi(t) 
\int_{\R^d} x_i x_j
\big \langle
 \nabla \eta_R(x),v_t(x)
\big\rangle
\,d\rho_t(x)\,dt
\right|\notag \\ 
&\leq
\int_{0}^{T}
|\xi(t)|
\int_{\R^d}
|x_i|\,|x_j|\,\|\nabla \eta_R(x)\|\,\|v_t(x)\|
\,d\rho_t(x)\,dt  \\
&\lesssim
\int_{0}^{T}
|\xi(t)|
\int_{\cB_{3R/2}\setminus\cB_{R}}
|x_i|\,R\cdot \frac{1}{R}\,\|v_t(x)\|
\,d\rho_t(x)\,dt  \\
&\lesssim_{\xi}
\int_{0}^{T}
\int_{\cB_{3R/2}\setminus\cB_{R}}
\|x\|\,\|v_t(x)\|
\,d \rho_t(x).
\end{align*}
By inequality \eqref{eq: lower bound} and \cref{lem:kin_en},
\[\int_{0}^{T}
\int_{\R^d}
\|x\|\,\|v_t(x)\|
\,d \rho_t(x) < \infty.
\]
As such, one may apply the dominated convergence theorem to conclude that 
\begin{align*}
\lim_{R \to \infty} \int_{0}^{T}
\int_{\cB_{3R/2}\setminus\cB_{R}}
\|x\|\,\|v_t(x)\|
\,d \rho_t(x) = 0,
\end{align*}
from which it follows that 
\begin{align}
\lim_{R \to \infty} \left|
\int_{0}^{T}
-\xi(t)
\int_{\R^d} x_i x_j
\big\langle
 \nabla \eta_R(x),v_t(x)
\big\rangle
\,d\rho_t(x)\,dt
\right| = 0.\label{eq:vanish_term}
\end{align}
Substituting \eqref{eq:term_1}--\eqref{eq:vanish_term} into \eqref{eq:moment-cutoff-identity},
we obtain 
\[
\int_{0}^{T}\xi'(t)\underbrace{\int_{\mathbb{R}^{d}}x_i x_j
\,d\rho_t(x)\,dt}_{=(\Sigma_{t})_{ij}}=
-\int_{0}^{T}\xi(t)\int_{\mathbb{R}^{d}}
\big(x_j v_t^{(i)} (x) + x_i v_t^{(j)}(x) \big)
\,d\rho_t(x)\,dt.
\]
This shows that $t \mapsto (\Sigma_t)_{ij}$ is locally absolutely continuous with 
\[
    \frac{d}{dt} (\Sigma_t)_{ij} = \int_{\mathbb{R}^{d}}
\big(x_j v_t^{(i)} (x) + x_i v_t^{(j)}(x) \big)
\,d\rho_t(x) \text{ for a.e. $t \in [0,\infty)$}.
\]

Substituting the expression for $v_t$ from \eqref{eq:v_t_closed}, we obtain for a.e. $t \in [0,\infty)$,
\[
\begin{aligned}
\frac{d}{dt}(\Sigma_t)_{ij}
&=
-\frac12 \sum_{k=1}^d (\Sigma^{-1}_t)_{ik}
    \int_{\mathbb R^d} x_j \partial_k \rho_t(x)\,dx
-\frac12 \sum_{k=1}^d (\Sigma^{-1}_t)_{jk}
    \int_{\mathbb R^d} x_i \partial_k \rho_t(x)\,dx \\
&\quad -\frac14 \sum_{k=1}^d (\Sigma_t^{-1}+\Sigma_t^{-2})_{ik}
    \int_{\mathbb R^d} x_j x_k\,d\rho_t(x)
-\frac14 \sum_{k=1}^d (\Sigma_t^{-1}+\Sigma_t^{-2})_{jk}
    \int_{\mathbb R^d} x_i x_k\,d\rho_t(x)\\
&= \frac12 \sum_{k=1}^d (\Sigma^{-1}_t)_{ik}\delta_{jk} +
\frac12 \sum_{k=1}^d (\Sigma^{-1}_t)_{jk}\delta_{ik} -\frac14 \sum_{k=1}^d (\Sigma_t^{-1}+\Sigma_t^{-2})_{ik} (\Sigma_t)_{jk}\\
&\quad -\frac14 \sum_{k=1}^d (\Sigma_t^{-1}+\Sigma_t^{-2})_{jk}(\Sigma_t)_{ik}\\
&= (\Sigma^{-1}_t)_{ij} - \frac{1}{4}\left(I+\Sigma_t^{-1}\right)_{ij} - \frac{1}{4}\left(I+\Sigma_t^{-1}\right)_{ij}\\
&= \frac{1}{2}\left(\Sigma^{-1}_t - I \right)_{ij},
\end{aligned}
\]
where we used the fact that $\int_{\R^d} x_j \partial_k \rho_t (x) \, dx = -\delta_{jk}$; cf. the proof of \cref{thm:nonlin-FP}. We conclude that $t \mapsto \Sigma_t$ solves the matrix ODE \eqref{eq:IVP}.

Finally, we shall show that the matrix ODE \eqref{eq:IVP} has a unique solution. Denote by $\mathbb S^d$ the vector space of symmetric matrices in $\R^{d \times d},$ and let $\mathbb S^d_{++}$ denote the corresponding cone of positive definite matrices. Define $\Psi\colon\mathbb{S}^d\to(-\infty,\infty]$ by
\[ 
\Psi(X) \coloneqq
    \begin{cases}
        \dfrac12\bigl(-\log\det(X)+\tr(X)\bigr),
        & X\in\mathbb{S}_{++}^d,\\[1ex]
        +\infty,
        & X\notin\mathbb{S}_{++}^d,
    \end{cases}
\]
which is proper, lower semicontinuous, and convex.
Its gradient is 
\[
    \nabla\Psi(X)
    =
    \frac12\left(-X^{-1}+I\right), \ X\in\mathbb{S}_{++}^d;
\]
see Exercise 21 in Chapter 3.1 of \cite{borwein2006convex}.
Thus, the matrix ODE \eqref{eq:IVP} can be written as the gradient flow
\[
    \dot X_t=-\nabla\Psi(X_t).
\]
Proposition~2.1 of \cite{santambrogio2017euclidean} yields
the desired uniqueness.

It remains to establish the uniqueness claim for the linear FPE \eqref{eq:linear-FP}. The uniform control on the eigenvalues of $\Sigma_t$ from \eqref{eq: lower bound} yields uniform control on the coefficients of \eqref{eq:linear-FP}. We now apply Theorem 9.4.6 in \cite{bogachev2022fokker} by choosing the Lyapunov function $V(x) = 1 + \|x\|^2$ to conclude the proof.
\end{proof}

\subsection{Proof of \cref{thm:lin-SDE}}
\label{pf:SDE_rep}
Since the smallest eigenvalues of $A_t = \Sigma_t$ are bounded away from zero uniformly over $t \in [0,\infty)$ (see \eqref{eq: lower bound}), the SDE (\ref{eq:linear-SDE}) admits a unique strong solution (cf. Theorems 5.2.5 and 5.2.9 in \cite{karatzas2014brownian}). Uniqueness in law follows by the Yamada--Watanabe theorem (Proposition 1 in \cite{yamada1971uniqueness}). 

Let $(X_t)_{t \in [0,\infty)}$ be the strong solution to (\ref{eq:linear-SDE}) with  $\nu_t = \text{Law}(X_t)$ for all $t \in [0,\infty)$. Applying Itô's formula  to smooth, compactly supported test functions, we obtain that $(\nu_t)_{t \in [0,\infty)}$ satisfies the FPE (\ref{eq:linear-FP}) with $\nu_0 = \rho_0$. By the uniqueness established in \cref{thm:lin-FP}, we conclude that $\nu_t = \rho_t$ for all $t \in [0,\infty)$. \qed

\subsection{Proofs of Corollaries \ref{cor:closed-form} and \ref{cor:gauss}}\label{pf:cor:closed-form}
These follow directly from Section 5.6 in \cite{karatzas2014brownian}.
\qed

\subsection{Proof of \cref{cor:regularity}}\label{pf:cor:regularity}
The matrix $Q_t$ is nonsingular for $t > 0$, so that
\[
    \rho_t(x) =  \frac{1}{(2\pi)^{d/2}\det(Q_t)^{1/2}}\int_{\R^d}
   \exp\left(-\frac12(x-\Phi_ty)^\intercal Q_t^{-1}(x-\Phi_ty)\right)
    \mathcal{\rho}_0(y)\,dy.
\]
Recall that $B \mapsto B^{-1}$ is smooth on $\mathbb S^{d}_{++}$. Since $A_t$ satisfies
$\dot A_t = \frac{1}{2}\left(A^{-1}_t - I\right)$, 
by Lemma 2.3 in \cite{teschl2012ode}, we have that $t \mapsto A_t$ is smooth. Likewise, 
 $t \mapsto \Phi_t$ is smooth. This shows that $t \mapsto Q_t$ is smooth. The desired claim follows easily. 
 \qed

\subsection{Proof of \cref{thm:exp_conv}}
\label{pf:exp_conv}
Recall the log-Sobolev inequality for the standard Gaussian distribution (cf. Proposition 6~(iv) in \cite{ambrosio2007gradient}),
\[
\sH(\rho \| \gamma) \le \frac{1}{2} \sI(\rho \| \gamma),
\]
where $\sI(\rho \| \gamma)$ denotes the relative Fisher information defined in \eqref{eq: fisher}.

We use Proposition 6.3 (v) from \cite{ambrosio2007gradient}. To this end, we shall verify its hypothesis. Fix any $T > 0$. By \cref{lem:kin_en}, 
$\int_0^T\int_{\R^d} \|v_t\|^2\,d\rho_t\,d t < \infty$. 
It remains to show that the integrated relative Fisher information is finite. Recall from the proof of \cref{thm:nonlin-FP} that $\cL_t[v_t] = - ( \frac{\nabla \rho_t}{\rho_t} + \id)$. As such, 
\begin{align*}
\int_0^T \mathsf{I}(\rho_t \| \gamma)\,d t
=
\int_0^T \int_{\R^d}
\left\|
\frac{\nabla \rho_t}{\rho_t} + y
\right\|^2
\,d \rho_t(y)\,d t  = \int_0^T \big \| \cL_t[v_t] \big \|_{L^2(\rho_t;\R^d)}^2 dt. 
\end{align*}
Using inequality \eqref{eq: lower bound} and the definition of $\cL_t$, one has 
\[
\int_0^T \big \| \cL_t[v_t] \big \|_{L^2(\rho_t;\R^d)}^2 dt \lesssim_{\rho_0} \int_0^T \| v_t \|_{L^2(\rho_t;\R^d)}^2 \, dt < \infty
\]
by \cref{lem:kin_en}. We now apply Proposition 6.3~(v) from \cite{ambrosio2007gradient} to conclude that the mapping $t \mapsto \sH(\rho_t \| \gamma)$ is locally absolutely continuous with
\[
    \frac{d}{d t}\sH(\rho_t \| \gamma)
    =
    \int_{\R^d}
    \left\langle
        v_t(y),\frac{\nabla\rho_t}{\rho_t}(y)+y
    \right\rangle
    d\rho_t.
\]
Recall that $v_t =- (\cL_t|_{\cI_t})^{-1}w_t$ with $w_t = \frac{\nabla \rho_t}{\rho_t} + \id$ for a.e. $t \in [0,\infty)$ (cf. the proof of \cref{thm:nonlin-FP}), so that the right-hand side agrees with 
$-\llangle (\cL_t|_{\cI_t})^{-1}[w_t],w_t \rrangle_{L^2(\rho_t;\R^d)}$. 
Observe that for $v \in \cI_t$ with $\| v \|_{L^2(\rho_t;\R^d)} = 1$, 
\[
\begin{split}
\llangle v,  \cL_{t}[v]\rrangle_{L^2(\rho_t;\RR^d)} &= 2\left ( \int_{\R^d} v^\intercal \Sigma_t v \, d\rho_t + \left \|  \int_{\R^d} xv(x)^\intercal \, d\rho_t(x) \right \|_{\F}^2 \right ) \in \big [ 2\lambda_{\min} (\Sigma_t), 4\lambda_{\max}(\Sigma_t) \big],
\end{split}
\]
where the upper bound follows as, for any unit vector $w \in \R^d$ and $i \in \{1,\dots,d \}$,
\[
\left  | \int_{\R^d}(w^\intercal x)  v^{(i)}(x)  \, d\rho_t(x)  \right|^2 \le (w^\intercal \Sigma_t w) \int_{\R^d} |v^{(i)}|^2 \, d\rho_t. 
\]
This implies that the spectrum of $\cL_{t}|_{\cI_t}$ is contained in $[ 2\lambda_{\min} (\Sigma_t), 4\lambda_{\max}(\Sigma_t) ]$ (cf. Proposition 6.9 in \cite{brezis2011functional}). Using spectral decomposition, we obtain
\[
-\llangle (\cL_t|_{\cI_t})^{-1}[w_t],w_t \rrangle_{L^2(\rho_t;\R^d)} \le -\frac{1}{4\lambda_{\max}(\Sigma_t)} \| w_t \|_{L^2(\rho_t;\R^d)}^2. 
\]

One can upper bound $\lambda_{\max}(\Sigma_t)$ by $4\sH(\rho_0 \| \gamma)+2d$ from  \eqref{eq: lower bound}, which can be sharpened 
 using the matrix ODE \eqref{eq:IVP} as follows.
Diagonalizing $\Sigma_0$ and using uniqueness of the matrix ODE, we obtain
\[
\Sigma_t=
Q\operatorname{diag}\bigl(\lambda_1(t),\ldots,\lambda_d(t)\bigr)Q^\intercal
\]
for some orthogonal matrix $Q,$ and each eigenvalue satisfies the ODE
\[
\dot{\lambda}_i(t) = \frac{1}{2}\left(\frac{1}{\lambda_i(t)}-1\right).
\]
By elementary arguments, each $\lambda_i(t)$ is monotone toward $1$, and hence
\[\lambda_{\max}(\Sigma_t) \le \max\left\{\lambda_{\max}(\Sigma_0),1\right\}.\]

Putting everything together, we conclude
\begin{align*}
    \frac{d}{d t}\sH(\rho_t \| \gamma)
\le - \frac{1}{4\max\{\lambda_{\max}(\Sigma_0),1\}}
\mathsf{I}(\rho_t \| \gamma) 
\le
- \frac{1}{2\max\{\lambda_{\max}(\Sigma_0),1\}}
\sH(\rho_t \| \gamma),
\end{align*}
where the last inequality follows by the log-Sobolev inequality. The desired result now follows from Gr\"{o}nwall's inequality. 
\qed

\section{Discussion and Concluding Remarks}
\label{sec: discussion}
This work studied the existence of an IGW gradient flow for the $\KL$ functional and provided a detailed characterization of its dynamics. We first showed that the gradient flow satisfies a PIDE and, by evaluating the action of the inverse mobility operator, obtained a more explicit description as a  nonlinear FPE. We then showed that the second moment dynamics decouple from the evolving law and satisfy an autonomous matrix ODE, which reduced the nonlinear FPE to a linear, time-inhomogeneous FPE. 
This, in turn, yielded a probabilistic representation of the IGW gradient flow as the time-marginal flow of a linear SDE resembling a second moment-modulated version of the classical OU process. Although the resulting dynamics share some qualitative properties with the OU process, including exponential convergence to equilibrium, they are not, in general, a deterministic time change of the OU process.

Several compelling directions for future work arise from these results. First, one may consider relative entropy w.r.t. a more general rotationally invariant log-concave measure and investigate whether the existence theory developed here, as well as analogous PDE and probabilistic representations, continue to hold. In this broader setting, one may still expect the IGW gradient flow to admit a nonlinear FPE representation, but the second moment decoupling that enables the linearization in the Gaussian case is unlikely to persist. The resulting dynamics may therefore genuinely depend on the evolving law, suggesting a McKean--Vlasov-type SDE representation. Second, it would be interesting to study IGW gradient flows associated with other rotationally invariant $f$-divergences, in analogy with corresponding developments in the Wasserstein setting \cite{liu2023minimizing,chewi2020svgd} . Finally, it remains an open question whether the FPE or SDE identified here arises naturally as a model for a physical or biological phenomenon.

\section{Acknowledgements}
The authors used ChatGPT 5.6 Sol plus to calibrate numerical values used in the counterexample in \cref{sec: counterexample}. They independently verified the argument therein. 

\appendix 
\label{sec: appendix}

\section{Proof of \cref{prop:fredholm}}
\label{sec: app A}
\underline{Part (i)}. We will only verify the kernel. All other claims are straightforward. Suppose that $\cL_\mu[v]=0$. Then
\[
v(x)=\underbrace{-\Sigma_\mu^{-1}\int_{\mathbb R^d} y\,v(y)^\intercal\,d\mu(y)}_{=:B} x.
\]
Substituting this back gives $B^\intercal x=-Bx$. 

Conversely, if $v(x)=Bx$ with $B^\intercal=-B$, then
\[
\cL_{\mu}[v](x) =
2\Sigma_\mu Bx+2 \left (\int_{\mathbb R^d} y(By)^\intercal\,d\mu(y) \right )x =
2\Sigma_\mu(B+B^\intercal)x
=
0.
\]

\medskip

\underline{Part (ii)}.  
Recall that for a given $A \in \R^{d \times d}$, the Sylvester equation $\Sigma_\mu B + B \Sigma_\mu = A$ admits the unique solution, $B = \int_0^\infty e^{-t\Sigma_\mu} A e^{-t\Sigma_\mu} \, dt$, which is symmetric if $A$ is symmetric (cf. \cite{simoncini2016computational}).

For a given $w \in \cI_\mu$, let $v$ be the right-hand side of \eqref{eq: inverse}. One can directly verify that $v \in \cI_\mu$ and $\cL_\mu [v]=w$, which shows surjectivity. Injectivity follows from $\mathrm{ker}(\cL_\mu) \cap \cI_\mu = \{ 0 \}$. Boundedness of the inverse follows from Corollary 2.7 in \cite{brezis2011functional}, or more directly from the expression $B_w = \frac12\int_0^\infty e^{-t\Sigma_\mu} A_w e^{-t\Sigma_\mu} \, dt$ with $A_w = \int_{\R^d} y w(y)^\intercal \, d\mu (y)$.
\qed

\section{Failure of $\lambda$-convexity along generalized IGW geodesics}
\label{sec: counterexample}
We shall first construct a generalized IGW geodesic along which $-(1/4)$-convexity fails for the $\sH(\cdot\|\gamma)$ functional. Throughout, we work on $\R^2.$ Let $\mu_0 = \mathcal N\left(
\begin{pmatrix}3\\0\end{pmatrix},
\begin{pmatrix}
1/16&0\\
0&9/4
\end{pmatrix}\right).$ For $r = \sqrt{286}$ and $L = \sqrt{1261225}$, we set
\begin{align*}
    \mu_1 &= \mathcal N\left(
\frac{1}{2L}\begin{pmatrix}921\\38r\end{pmatrix},
\frac{1}{(140L)^2}\begin{pmatrix}
46603&5184r\\
-816r&73497
\end{pmatrix}\begin{pmatrix}
46603&5184r\\
-816r&73497
\end{pmatrix}^\intercal\right),\\
\mu_2 &= \mathcal N\left(
\frac{1}{2L}\begin{pmatrix}921\\-38r\end{pmatrix},
\frac{1}{(140L)^2}\begin{pmatrix}
46603&-5184r\\
816r&73497
\end{pmatrix}\begin{pmatrix}
46603&-5184r\\
816r&73497
\end{pmatrix}^\intercal\right).
\end{align*}
Consider a random variable $\xi \sim \cN(0,I).$ Define
\begin{align*}
    X &= \begin{pmatrix}3\\0\end{pmatrix}+\begin{pmatrix}1/4&0\\0&3/2\end{pmatrix}\xi,\\
    Y &=
\frac1{2L}
\begin{pmatrix}
921\\
38r
\end{pmatrix}
+
\frac1{140L}
\begin{pmatrix}
46603&5184r\\
-816r&73497
\end{pmatrix}\xi,\\
Z
&=
\frac1{2L}
\begin{pmatrix}
921\\
-38r
\end{pmatrix}
+
\frac1{140L}
\begin{pmatrix}
46603&-5184r\\
816r&73497
\end{pmatrix}\xi.
\end{align*}
Since all orthogonal transformations in $\R^2$ are compositions of rotations and reflections, one may use the formula in Theorem 3.1 from \cite{dandapanthula2025optimal} to check that $(X,Y)$ and $(X,Z)$ are the optimal IGW couplings for $(\mu_0,\mu_1)$ and $(\mu_0,\mu_2)$, respectively. One may also verify numerically that the cross-covariance matrices for $(X,Y)$ and $(X,Z)$ are positive definite. Thus $\pi = \mathrm{Law}(X,Y,Z)$ is a valid glued coupling to construct the generalized IGW geodesic for $(\mu_0,\mu_1,\mu_2)$ (see Definition 3.2 in \cite{zhang2024gradient}).  
For the generalized IGW geodesic
\[
\nu_t
=
\big((y,z) \mapsto (1-t)y+tz\big)_\#\pi,
\]
its midpoint is the law of
\[
\frac{Y+Z}{2}
=
\frac1{2L}
\begin{pmatrix}
921\\
0
\end{pmatrix}
+
\frac1{140L}
\begin{pmatrix}
46603&0\\
0&73497
\end{pmatrix}\xi.
\]
For $(Y',Z')$ an independent copy
of $(Y,Z),$ a direct calculation  shows that 
\begin{equation}\label{eq:1/4}
    \sH(\nu_{1/2}\|\gamma)
>
\frac12\sH(\mu_1\|\gamma)+\frac12\sH(\mu_2\|\gamma)
+\frac1{16}\E_{\pi\otimes\pi}\left[ \left(\langle Y,Y'\rangle-\langle Z,Z'\rangle\right)^2\right],
\end{equation}
so that $\sH(\cdot\|\gamma)$ is not $-(1/4)$-convex along generalized IGW geodesics.
We shall now use a scaling argument to show that $\sH(\cdot\|\gamma)$ is not $\lambda$-convex along generalized IGW geodesics for any $\lambda \in \R.$

Fix $s \in (0,1)$ and consider $(sX,sY,sZ)$, whose law is denoted by $\pi^s$. It follows that  $(sX,sY)$ and $(sX,sZ)$ are optimal IGW couplings for the respective pairs of marginals and that they have positive definite cross-covariance matrices.  Thus the conditions for the IGW generalized geodesic are satisfied for this triple. Let
\[ 
\nu_t^s = \bigl((y,z) \mapsto (1-t)y+tz\bigr)_{\#}\pi^s.
\]

Recall that
\[ \sH\bigl(\mathcal N(m,\Sigma)\|\gamma\bigr) =
\frac12 \left[ \tr(\Sigma)+\|m\|^2 - \log\det(\Sigma)-d
\right],
\]
where we use $\Sigma$ to denote the covariance matrix (and not the second moment matrix) for the rest of this proof. Denote the mean and covariance matrix of $\nu_{1/2}$ by $m_{\nu_{1/2}}$ and $\Sigma_{\nu_{1/2}}$ respectively, and let $\mu^s_1 = \mathrm{Law}(sY)$ and $\mu^s_2 = \mathrm{Law}(sZ).$
Observe that
\begin{align*}
\sH\left(\nu_{1/2}^s\|\gamma\right)
&=\frac{s^2}{2}\left(\tr\left(\Sigma_{\nu_{1/2}}\right)
+\left\|m_{\nu_{1/2}}\right\|^2\right)
-\frac12\left(\log\det\left(\Sigma_{\nu_{1/2}}\right)
+\log s^4+d\right)\\
&= s^2\sH(\nu_{1/2}\|\gamma) + \frac{s^2}{2}(\log \det(\Sigma_{\nu_{1/2}}) + d) -\frac12\left(\log\det\left(\Sigma_{\nu_{1/2}}\right)
+\log s^4+d\right)\\
&> \frac{1}{2}\sH(\mu^s_1\|\gamma) + \frac{1}{2}\sH(\mu^s_2\|\gamma) + \frac{1}{16s^2}
\int
\left|
\langle y,y'\rangle-\langle z,z'\rangle
\right|^2
\,d\pi_{1,2}^s \otimes \pi_{1,2}^s(y,z,y',z')\\ 
&\quad + \frac{1}{4}(1 - s^2)\log\left(\frac{\det (\Sigma_{\mu_1} \Sigma_{\mu_2})}{\det (\Sigma_{\nu_{1/2}})^2}\right),
\end{align*}
where the inequality follows from adding and subtracting the necessary terms and applying (\ref{eq:1/4}). By a direct calculation, one may verify that $\det(\Sigma_{\mu_1} \Sigma_{\mu_2}) > \det (\Sigma_{\nu_{1/2}})^2$. As such, for $s < 1,$
\[
\sH\left(\nu_{1/2}^s \| \gamma \right)
>
\frac12\sH(\mu_1^s\|\gamma)
+\frac12\sH(\mu_2^s\|\gamma)
+\frac{1}{16s^2}
\int
\left|
\langle y,y'\rangle-\langle z,z'\rangle
\right|^2
\,d\pi_{1,2}^s \otimes \pi_{1,2}^s(y,z,y',z').
\]
Thus, the scaled example violates $-\left(\frac{1}{4s^2}\right)$-convexity, and taking $s\to0$ concludes the argument.

\section{Proof of \cref{lem:prox_ineq}}\label{app:2}
We recall the following local convexity result for $\IGW^2$, taken from Section 10.1 in \cite{zhang2024gradient}.

\begin{lemma}\label{lem:IGW_conv}
    In the setting of \cref{def:MGG}, it holds that
    \begin{multline*}
    \IGW^2(\nu_t,\mu_0) 
    = (1-t)\IGW^2(\mu_1,\mu_0) + t\IGW^2(\mu_2,\mu_0) \\
    - t\left(1- \frac{8\sqrt{2}}{c_{1,\mu_0}}\IGW(\mu_0,\mu_1)\alpha_{\pi} \right)
     \int_{(\R^d \times \R^d)^2} \bigl|\langle y,y'\rangle-\langle z,z'\rangle\bigr|^2\, d\pi_{1,2}\otimes\pi_{1,2}(y,z,y',z')\\
     +  t\frac{12\sqrt{2}}{c_{1,\mu_0}}\alpha_{\pi} \IGW^3(\mu_0,\mu_1)  + O(t^2),
\end{multline*}
where $\alpha_{\pi} = 1+ \frac{c_{2,\pi}^4}{2} M_2(\mu_1)M_2(\mu_0)+ c_{2,\pi} M_2(\mu_1),$ and $c_{1,\mu_0},$ and $c_{2,\pi}$ are defined as in \cref{lem:local_convexity}
\end{lemma}

\begin{proof}[Proof of \cref{lem:prox_ineq}]
Observe that
\begin{align*}
&\sH(\mu_1 \| \gamma) + \frac{1}{2\tau} \IGW^2(\mu_1,\mu_0)\\
&\le \sH(\nu_t \| \gamma) + \frac{1}{2\tau} \IGW^2(\nu_t,\mu_0) \\
&\le (1-t)\sH(\mu_1 \| \gamma) + t\sH(\mu_2 \| \gamma)
+ t\frac{8\sqrt{2}}{c_{1,\mu_0}}\beta_\pi
\int_{(\R^d \times \R^d)^2} \bigl|\langle y,y'\rangle-\langle z,z'\rangle\bigr|^2\, d\pi_{1,2}\otimes\pi_{1,2}(y,z,y',z')\\
&\quad + t\frac{12\sqrt{2}}{c_{1,\mu_0}}\beta_\pi \IGW^2(\mu_0,\mu_1)
   + (1-t)\frac{\IGW^2(\mu_1,\mu_0)}{2\tau} + t\frac{\IGW^2(\mu_2,\mu_0)}{2\tau} \\
&\quad - \frac{t}{2\tau}
\left(1 - \frac{8\sqrt{2}}{c_{1,\mu_0}} \IGW(\mu_0,\mu_1)\alpha_\pi \right)
\int_{(\R^d \times \R^d)^2} \bigl|\langle y,y'\rangle-\langle z,z'\rangle\bigr|^2\, d\pi_{1,2}\otimes\pi_{1,2}(y,z,y',z') \\
&\quad + \frac{t}{2\tau}\frac{12\sqrt{2}}{c_{1,\mu_0}}\alpha_\pi
\IGW^3(\mu_0,\mu_1)
+ O(t^2),
\end{align*}
where the first inequality follows from the optimality of $\mu_1$, and the second follows from the local convexity of $\KL$ and $\IGW^2$, see \cref{lem:local_convexity} and \cref{lem:IGW_conv}. Rearranging terms, dividing both sides by $t$, and taking the limit $t \to 0$, we obtain the desired inequality. 
\end{proof}

\section{Proof of \eqref{eq: cauchy 2}}
\label{app: 3}
\underline{Step 1}. 
We shall apply \cref{lem:prox_ineq} with $(\mu_0,\mu_1,\mu_2) = (\rho^{\tau}_{i-1},\rho^{\tau}_i,\nu)$ for $\nu$ close to $\rho_0$ in $\IGW$. We first verify the existence of a modified generalized geodesic, which, by Lemma 4.2~(ii) in \cite{zhang2024gradient}, holds whenever $\rho^{\tau}_{i-1},\rho^{\tau}_i,\nu \in \cB_{\IGW}\left(\rho_0, \bar \delta\right)$ with $\bar \delta = \frac{(1-1/\sqrt{2})\lambda_{\min}(\Sigma_{\rho_0})}{2^{1/4}}$, where 
$\cB_{\IGW}(\rho_0,r)$ denotes the open IGW ball
centered at $\rho_0$ with radius $r$,
\[ \cB_{\IGW}(\mu,r)
\coloneqq
\left\{
    \nu\in\cP_2(\R^d): \IGW(\mu,\nu)<r
\right\}.
\]
 Observe that for all $j \in \{ 1,\dots, n\}$,
\begin{equation}\label{eq:ball_ineq}
\begin{split}
    \IGW(\rho_j^\tau,\rho_0) &\leq \sum_{i=0}^{j-1}\IGW(\rho_{i+1}^\tau,\rho_i^\tau) \leq \left(\sum_{i=0}^{j-1} \tau \right)^{1/2}\left(\sum_{i=0}^{j-1} \frac{1}{\tau} \IGW^2(\rho_{i+1}^\tau,\rho_i^\tau)) \right)^{1/2}\\ 
    &\leq \sqrt{\delta}
\left(\sum_{i=0}^{j-1} \frac{1}{\tau} \IGW^2(\rho_{i+1}^\tau,\rho_i^\tau)) \right)^{1/2} \leq
\left(2 \delta \sum_{i=0}^{j-1} \sH(\rho_i^\tau\|\gamma) - \sH(\rho_{i+1}^\tau\|\gamma)\right)^{1/2} \\ 
&\leq \sqrt{2\delta\sH(\rho_0\|\gamma)},
\end{split}
\end{equation}
which implies that  $\rho_j^\tau \in \cB_{\IGW}\left(\rho_0, \bar \delta \right)$ for all $j \in \{ 1,\dots,n\}$ (recall our choice of $\delta$).
As such, the existence of a modified generalized geodesic is guaranteed for $(\mu_0,\mu_1,\mu_2) = (\rho^{\tau}_{i-1},\rho^{\tau}_i,\nu)$ for $\nu\in\cB_{\IGW}\left(\rho_0, \bar \delta\right)$. 

We observe that whenever
\begin{equation} \label{eq:non_neg}
    1 - \frac{8\sqrt{2}}{c_{1,\mu_0}} \, \IGW(\mu_0,\mu_1)\, \alpha_\pi
- 2\tau \frac{8\sqrt{2}}{c_{1,\mu_0}}\,\beta_\pi \geq 0,
\end{equation}
the inequality from \cref{lem:prox_ineq} translates to  
\begin{equation}\label{eq:prox_ineq2}
\begin{split}
&\sH(\mu_1\|\gamma) + \frac{1}{2\tau} \, \IGW^2(\mu_1,\mu_0) \\
&\le
\sH(\mu_2\|\gamma)
- \frac{1}{2\tau}\Bigg(
\left(
1 - \frac{8\sqrt{2}}{c_{1,\mu_0}} \, \IGW(\mu_0,\mu_1)\, \alpha_\pi
- 2\tau \frac{8\sqrt{2}}{c_{1,\mu_0}}\,\beta_\pi
\right)\
 \IGW^2(\mu_1,\mu_2) \\[6pt]
&\qquad
- 2\tau \frac{12\sqrt{2}}{c_{1,\mu_0}}\,\beta_\pi \,
\IGW^2(\mu_0,\mu_1)
- \IGW^2(\mu_2,\mu_0)
- \frac{12\sqrt{2}}{c_{1,\mu_0}}\,\alpha_\pi \,
\IGW^3(\mu_0,\mu_1)\Bigg)
\end{split}
\end{equation}
by the definition of the IGW distance. 

We shall verify that condition (\ref{eq:non_neg}) holds whenever the step size $\tau$ is small enough. Observe that
\begin{align*}
    \IGW^2(\rho^{\tau}_{i-1},\rho^{\tau}_i) \leq 4 \tau \sH(\rho_0\|\gamma)
\end{align*}
from the definition of the JKO scheme. 
In addition, one can find constants $c_{1,\rho_0},\beta_{\rho_0}$, and $\alpha_{\rho_0}$ that  depend only on  $\rho_0$ such that $c_{1,\mu_0} \geq c_{1,\rho_0},$ $\alpha_{\pi} \leq \alpha_{\rho_0}$, and $\beta_{\pi} \leq \beta_{\rho_0}$.
Indeed, by \eqref{eq: lower bound}, 
$c_{1,\mu_0} \geq k_{\rho_0}$. 
For $\alpha_{\pi}$ and $\beta_{\pi}$, observe that 
\[
M_2(\mu_0), M_2(\mu_1) \leq 4d\sH(\rho_0\|\gamma)+2d^2
\]
by \eqref{eq: lower bound}.  Since $\nu \in \cB_{\IGW}(\rho_0,\bar \delta)$, by Lemma 4.2 in \cite{zhang2024gradient} and \eqref{eq: lower bound},
 \[
 \lambda_{\min}(A_i) \geq \frac{k_{\rho_0}}{2\sqrt{2}} \qquad \text{ for $i =1,2$}.
 \] 
Recalling that $\lambda_{\max}(A_i) \leq \sqrt{\lambda_{\max}(\Sigma_{\rho^\tau_{i - 1}})\lambda_{\max}(\Sigma_{\nu})}$, which can be upper bounded solely in terms of the eigenvalues of $\Sigma_{\rho_0}$ and the value $\sH(\rho_0\|\gamma)$ by Lemma 4.2 in \cite{zhang2024gradient}, one obtains $\alpha_\pi \le \alpha_{\rho_0}$ and $\beta_\pi \le \beta_{\rho_0}$ for some constants $\alpha_{\rho_0}$ and $\beta_{\rho_0}$ that depend only on $\rho_0$.

As such, whenever 
\[\sqrt{\tau} \leq \min\left\{ \frac{c_{1,\rho_0}}
        {16\sqrt{2}\left(\sqrt{\mathsf{H}(\rho_0\Vert\gamma)}\,\alpha_{\rho_0} +\beta_{\rho_0}\right)},1\right\}
\]
condition (\ref{eq:non_neg}) holds, which guarantees inequality \eqref{eq:prox_ineq2}.

\medskip

\underline{Step 2}. 
We now use inequality  \eqref{eq:prox_ineq2} to find an upper bound on the time derivative of $d_{\tau \eta}^2(t,t)$. We define some notation to simplify the subsequent derivation. For $i=1,\ldots,n$ and $t\in((i-1)\tau,i\tau]$, set 
\[
\sD_\tau(t) \coloneqq \IGW(\rho^\tau_{i-1},\rho^\tau_i) \quad \text{and} \quad \sigma_\tau(t)
\coloneqq
-\frac{4\sqrt{2}}{\tau c_{1,\rho_0}}
\sD_\tau(t)\alpha_{\rho_0}
-\frac{8\sqrt{2}}{c_{1,\rho_0}}
\beta_{\rho_0}.
\]
Recalling the weighting function $\ell_\tau$, define the interpolations
\begin{align*}
\sH_\tau(t) &\coloneqq (1-\ell_\tau(t)) \sH(\rho^\tau_{i-1}\|\gamma) + \ell_\tau(t) \sH(\rho^\tau_i\|\gamma),\\
\sR_\tau(t) &\coloneqq \sH_\tau(t) - \sH(\rho^\tau_i\|\gamma),
\end{align*}
where $i=1,\ldots,n$ and $t\in((i-1)\tau,i\tau]$.
With this notation, we have
\begin{equation}\label{eq:1}
\begin{split}
     \sH(\rho^\tau_i\|\gamma) + \frac{\sD^2_{\tau}(t)}{2\tau} &\leq \sH(\nu\|\gamma) + \frac{1}{2 \tau} \IGW^2(\nu,\rho^\tau_{i -1}) - \left(\frac{1}{2\tau} + \sigma_{\tau}(t)\right) \IGW^2(\rho^\tau_i,\nu)\\ 
     &\quad + \frac{12 \sqrt 2}{c_{1,\rho_0}}\beta_{\rho_0}\sD^2_{\tau}(t) + \frac{1}{2\tau}\frac{12 \sqrt 2}{c_{1,\rho_0}}\alpha_{\rho_0}\sD^{3}_{\tau}(t).
     \end{split}
\end{equation}
When $\tau \leq 1/2$ and $\sD_{\tau}(t) \leq 1,$ it follows that
\begin{align*}
    \frac{12 \sqrt 2}{c_{1,\rho_0}}\beta_{\rho_0}\sD^2_{\tau}(t) + \frac{1}{2\tau}\frac{12 \sqrt 2}{c_{1,\rho_0}}\alpha_{\rho_0}\sD^{3}_{\tau}(t) &\leq \left(\frac{12 \sqrt 2}{c_{1,\rho_0}} \beta_{\rho_0} + \frac{12 \sqrt 2}{c_{1,\rho_0}} \alpha_{\rho_0}\right)\frac{\sD^2_{\tau}(t)}{2 \tau}\coloneqq \frac{\sD^2_{\tau}(t)}{2 \iota_{\rho_0} \tau}
\end{align*}
Observe that 
\[
\frac{d}{dt} d^2_\tau(t;\nu) = \frac{\IGW^2(\rho^\tau_i,\nu)-\IGW^2(\rho^\tau_{i-1},\nu)}{\tau},\quad t\in ((i-1)\tau,i\tau],
\]
which further simplifies \eqref{eq:1} into
\begin{align*}
    \frac{1}{2}\frac{d}{dt} d^2_\tau(t;\nu) + \sigma_\tau(t)\IGW^2(\rho_i,\nu) &\leq \sH(\nu\|\gamma) - \sH_{\tau}(t) + R_{\tau}(t) - \frac{\sD^2_{\tau}(t)}{2 \tau} + \frac{\sD^2_{\tau}(t)}{2 \iota_{\rho_0}\tau}\\
    & \leq \sH(\nu\|\gamma) - \sH_{\tau}(t) + R_{\tau}(t) + \frac{\sD^2_{\tau}(t)}{2 \iota_{\rho_0}\tau}.
\end{align*}
Note that $\sigma_{\tau}(t) \leq 0$ for all $t \in [0,\delta)$ and $\big|d^2_\tau(t;\nu)- \IGW^2\big(\bar{\rho}^n_t,\nu\big)\big|\leq 2d_\tau(t;\nu)\sD_\tau(t) + \sD^2_\tau(t).$ Insert this into the above to further obtain
\begin{align*}
    &\frac{1}{2}\frac{d}{dt} d^2_\tau(t;\nu) + \sigma_\tau(t) d^2_\tau(t;\nu) - 2|\sigma_\tau(t)|d_\tau(t;\nu)\sD_\tau(t)\\
    &\quad \leq |\sigma_\tau(t)|\sD^2_\tau(t) + \sH(\nu\|\gamma) - \sH_\tau(t) +  \sR_\tau(t) + \frac{\sD^2_{\tau}(t)}{2 \iota_{\rho_0}\tau}.
\end{align*}

Fix $t$, and consider a linear combination with coefficients $1-\ell_\eta(s),\ell_\eta(s)$ of the above inequality instantiated at $\nu=\rho^\eta_{j-1},\rho^\eta_j$, respectively. By the fact that $(1-\ell_\eta(s))d_\tau(t;\rho^\eta_{j-1}) + \ell_\eta(s)d_\tau(t;\rho^\eta_j)\leq d_{\tau\eta}(t,s)$, we have 
\begin{align*}
    &\frac{1}{2}\frac{d}{dt} d^2_{\tau\eta}(t,s) + \sigma_\tau(t) d^2_{\tau\eta}(t,s) + 2\sigma_\tau(t)d_{\tau\eta}(t,s)\sD_\tau(t)\\
    &\quad \leq -\sigma_\tau(t)\sD^2_\tau(t) + \sH_\eta(s) - \sH_\tau(t) +  \sR_\tau(t) + \frac{\sD^2_{\tau}(t)}{2 \iota_{\rho_0}\tau}.
\end{align*}
Switching the roles of $\tau$ and $\eta$ similarly yields
\begin{align*}
    &\frac{1}{2}\frac{d}{dt} d^2_{\eta\tau}(t,s) + \sigma_\eta(t) d^2_{\eta\tau}(t,s) + 2\sigma_\eta(t)d_{\eta\tau}(t,s)\sD_\eta(t)\\
    &\quad \leq -\sigma_\eta(t)\sD^2_\eta(t) + \sH_\tau(s) - \sH_\eta(t) +  \sR_\eta(t) + \frac{\sD^2_{\eta}(t)}{2 \iota_{\rho_0}\eta}.
\end{align*}
From the chain rule and since $d^2_{\tau\eta}(t,s) = d^2_{\eta\tau}(s,t),$ evaluating the two inequalities above at $s = t$ and then adding them yields
\begin{equation}
\label{eq:var_ineq}
\begin{split}
    &\frac{1}{2}\frac{d}{dt} d^2_{\tau\eta}(t,t) + (\sigma_\tau(t)+\sigma_\eta(t)) d^2_{\tau\eta}(t,t)  - 2\big(|\sigma_\tau(t)|\sD_\tau(t) + |\sigma_\eta(t)|\sD_\eta(t)\big) d_{\tau\eta}(t,t)\\
    &\quad \leq |\sigma_\tau(t)|\sD^2_\tau(t) +  |\sigma_\eta(t)|\sD^2_\eta(t) +  \sR_\tau(t)+  \sR_\eta(t) + \frac{\sD^2_{\tau}(t)}{2 \iota_{\rho_0}\tau} + \frac{\sD^2_{\tau}(t)}{2 \iota_{\rho_0}\eta}.
    \end{split}
\end{equation}
Here, we have used above that $\sigma_{\tau}(t),\sigma_{\eta}(t) \leq 0$ for all $t \in [0,\delta).$

\medskip

\underline{Step 3}.
Finally, we invoke the following Grönwall-type lemma.

\begin{lemma}[A version of Gr\"onwall {\cite[Lemma 4.4]{zhang2024gradient}}]\label{lem:gronwall}
    Let $a,b,c:[0,\delta]\to \RR$ be integrable functions and $x:[0,\delta]\to\RR$ be continuous, such that
    \begin{align*}
        \frac{d}{dt} x^2(t) + a(t)x^2(t)  \leq c(t)+ b(t)x(t),\quad \forall t\in[0,\delta].
    \end{align*}
    Denoting $G(t)\coloneqq\int_0^t a(s) ds$, for every $T\in[0,\delta]$, we have
    \begin{align*}
        e^{G(T)/2} \|x(T)\| \leq \left(x^2(0) + \sup_{t\in[0,T]} \int_0^t  e^{G(s)}c(s) ds \right)^{1/2} + 2\int_0^T\big|b(t) e^{G(t)/2}\big| dt. 
    \end{align*}
\end{lemma}
Apply \cref{lem:gronwall} to \eqref{eq:var_ineq} as follows. Set
 \begin{align*}
x(t) &= d_{\tau \eta}(t,t), \\
a(t) &= 2\big({\sigma}_\tau(t) + {\sigma}_\eta(t)\big), \\
b(t) &= 4\big(|{\sigma}_\tau(t)| \sD_\tau(t) + |{\sigma}_\eta(t)| \sD_\eta(t)\big), \\
c(t) &= 2(|{\sigma}_\tau(t)| \sD^2_\tau(t) + |{\sigma}_\eta(t)| \sD^2_\eta(t)
+ \sR_\tau(t) + \sR_\eta(t) + \frac{\sD^2_\tau(t)}{2\iota_{\rho_0} \tau} + \frac{\sD^2_\eta(t)}{2\iota_{\rho_0} \eta}).
\end{align*}
By inequality \eqref{eq:ball_ineq},
\[
\int_0^t \sD_{\tau}(s)\,ds
\leq 
\int_0^{\delta} \sD_{\tau}(s)\,ds = \sum^n_{i =1} \int^{i\tau}_{(i-1)\tau} \sD_{\tau}(s)\,ds= \sum_{i=1}^n \tau \IGW(\rho^{\tau}_i,\rho^{\tau}_{i-1}) \leq 2\tau\sqrt{\delta\sH(\rho_0\|\gamma)}.
\]
Similarly,
\[
\int_0^t \sD^2_{\tau}(s)\,ds
\leq 
\int_0^{\delta} \sD^2_{\tau}(s)\,ds = \sum^n_{i =1} \int^{i\tau}_{(i-1)\tau} \sD^2_{\tau}(s)\,ds\\
= \sum_{i=1}^n \tau \IGW^2(\rho^{\tau}_i,\rho^{\tau}_{i-1}) \leq 2\tau^2\sH(\rho_0\|\gamma)
\]
and
\begin{align*}
    \bigg|\int^t_0 \sR_{\tau}(s) ds \bigg| &\leq \int_0^{\delta} \sR_{\tau}(s)\,ds \leq \bigg| \sum_{i = 1}^n\int^{i\tau}_{(i-1)\tau} \sH(\rho^{\tau}_i\|\gamma) - \sH(\rho^{\tau}_{i-1}\|\gamma) ds \bigg | \leq \tau\sH(\rho_0\|\gamma).
\end{align*}
Similar bounds hold when replacing $\tau$ by $\eta$. As such, one has
\begin{align*}
\left|\int_0^t a(s)\,ds\right| \lesssim_{\rho_0} 1, \quad 
\left|\int_0^t b(s)\,ds\right| \lesssim_{\rho_0} \tau + \eta, \quad
\left|\int_0^t c(s)\,ds\right| \lesssim_{\rho_0} \tau + \eta.
\end{align*}
From \cref{lem:gronwall}, for all $\tau ,\eta$ sufficiently small,
\begin{align*}
\|x(T)\|
&\lesssim_{\rho_0} d_{\tau \eta}(0,0)
+ (\tau +\eta)^{1/2}
+ (\tau +\eta),
\end{align*}
and for all $T \in [0,\delta]$,
\[
d_{\tau \eta}(T,T) \lesssim_{\rho_0} (\tau+\eta)^{1/2}
+ 2(\tau+\eta),
\]
which gives us the desired estimate. \qed

\bibliographystyle{alpha}
\bibliography{references}
\end{document}